\documentclass[11pt, letterpaper]{amsart}

\usepackage{amssymb,enumitem,mleftright,parskip, tikz}
\usepackage{enumitem,mleftright,parskip}
\usepackage[margin = 1 in]{geometry}
\usepackage{dsfont}
\usepackage{amsmath}
\usepackage{amsfonts}
\usepackage{amssymb}
\usepackage{bbm}
\usepackage{mathrsfs}
\usepackage{mathtools}
\usepackage{amsthm}
\usepackage{verbatim}
\usepackage{upgreek}
\usepackage{tikz-cd}
\usepackage{enumitem}
\usepackage{dsfont}

\usepackage[alphabetic]{amsrefs}

\usepackage[utf8]{inputenc}

\setlist[itemize]{leftmargin = *}
\setlist[enumerate]{leftmargin = *}

\allowdisplaybreaks

\newtheorem{Prop}{Proposition}[section]
\newtheorem{Thm}[Prop]{Theorem}

\newtheorem{corollary}[Prop]{Corollary}

\theoremstyle{definition}

\newtheorem{Def}[Prop]{Definition}
\newtheorem*{remark}{Remark}

\newtheorem{example}[Prop]{Example}

\newcommand{\G}{\mathbb{G}}

\newcommand{\GG}{\mathbb{G}}
\newcommand{\HH}{\mathbb{H}}

\newcommand{\Br}[1]{\mleft( #1 \mright)}

\newcommand{\op}{\textnormal{op}}

\newcommand{\Comm}[1]{#1^{\c}}

\newcommand{\Trip}[3]{\Br{#1,#2,#3}}

\newcommand{\id}{\operatorname{id}}
\newcommand{\wh}{\widehat}

\newcommand{\mc}{\mathcal}

\newcommand{\vp}{\varphi}
\newcommand{\ov}{\overline}
\newcommand{\oon}{\operatorname}
\newcommand{\lin}{\oon{span}}

\newcommand{\la}{\langle}
\newcommand{\ra}{\rangle}

\renewcommand{\c}{\textnormal{c}}

\newcommand{\vv}{\mathrm{V}}
\newcommand{\ww}{\mathrm{W}}

\newcommand{\Ww}{\mathds{W}}
\newcommand{\wW}{\text{\reflectbox{$\Ww$}}\:\!}
\newcommand{\vV}{\text{\reflectbox{$\mathds{V}$}}\:\!}

\newcommand{\cs}{\ensuremath{C^*}}
\newcommand{\bb}{\mathbb{B}}
\newcommand{\br}{\mathbb{R}}
\newcommand{\bc}{\mathbb{C}}
\newcommand{\bk}{\mathbb{K}}

\newcommand{\bZ}{\mathbb{Z}}

\newcommand{\bdops}{\bb}

\renewcommand{\id}{\text{id}}

\renewcommand{\ker}{\text{ker}}

\newcommand{\inv}{^{-1}} 

\definecolor{magenta}{rgb}{0.9, 0, 0.9}
\definecolor{coral}{rgb}{1, 0.25, 0.25}
\definecolor{turquoise}{rgb}{0.25, 0.88, 0.82}

\title{The Covariant Stone--von Neumann Theorem\\ for Locally Compact Quantum Groups}

\author[Hall]{Lucas Hall}
\address{Department of Mathematics
\\Michigan State University
\\East Lansing, Michigan 48824
\newline
Department of Mathematics
\\ University of Haifa
\\ Haifa, Israel 3498838}
\email{hallluc1@msu.edu}

\author[Huang]{Leonard Huang}
\address{School of Physical and Mathematical Sciences \\ Nanyang Technological University \\ 21 Nanyang Link, Singapore 637371}
\email{leonard.huang@ntu.edu.sg}

\author[Krajczok]{Jacek Krajczok}
\address{Vrije Universiteit Brussel\\
Pleinlaan 2, 1050 Brussels, Belgium}
\email{jacek.krajczok@vub.be}

\author[Tobolski]{Mariusz Tobolski}
\address{Instytut Matematyczny, Uniwersytet Wroc\l{}awski, 
pl. Grunwaldzki 2/4, 50-384 Wroc\l{}aw, Poland
}
\email{mariusz.tobolski@math.uni.wroc.pl}

\begin{document}

% ABSTRACT %%%%%%%%%%%%%%%%%%%%%%%%%%%%%%%%%%%%%%%%%%%%%%%%%%%%%%%%%%%%%%%%%%%%%%%%%%%%%%%%%%%%%%%%%%%%%%%%%%%%%%%%%%%%%%%%%%%%%%%%%%%%%%%%%%%%

\begin{abstract}

The Stone--von Neumann Theorem is a fundamental result 
which unified the competing quantum-mechanical models of 
matrix mechanics and wave mechanics. 
%It's mechanism of 
%proof ultimately involved the study of unitary group 
%representations on a Hilbert space. 
In this article, we 
continue the broad generalization set out by Huang and 
Ismert and by Hall, Huang, and Quigg, analyzing 
representations of locally compact quantum dynamical 
systems defined on Hilbert modules, of which the classical 
result is a special case. We introduce a pair of modular 
representations which subsume numerous models available
in the literature and, using the classical strategy of Rieffel,
 prove a Stone--von Neumann-type 
theorem for maximal actions of regular locally compact 
quantum groups on elementary C*-algebras. In particular, we generalize the  Mackey--Stone--von 
Neumann Theorem to regular locally compact 
quantum groups whose trivial actions on $\mathbb{C}$ are maximal and recover the multiplicity results of 
Hall, Huang, and Quigg. With this characterization in hand,
we prove our main result showing 
that if a dynamical system $(\mathbb{G},A,\alpha)$ satisfies the multiplicity assumption of 
the generalized Stone--von Neumann theorem, and if the coefficient algebra 
$A$ admits a faithful state, then the 
spectrum of the iterated crossed product 
$\widehat{\mathbb{G}}^{\rm op}\ltimes (\mathbb{G}\ltimes A)$ 
consists of a~single point. In the case of a separable coefficient algebra
or a regular acting quantum group, we further characterize features of this system, 
and thus obtain a~partial 
converse to the Stone--von Neumann theorem in the quantum 
group setting. As a~corollary, we show that a regular 
locally compact quantum group satisfies the generalized 
Stone--von Neumann theorem if and only if it is strongly regular.
\end{abstract} 

%%%%%%%%%%%%%%%%%%%%%%%%%%%%%%%%%%%%%%%%%%%%%%%%%%%%%%%%%%%%%%%%%%%%%%%%%%%%%%%%%%%%%%%%%%%%%%%%%%%%%%%%%%%%%%%%%%%%%%%%%%%%%%%%%%%%%%%%%%%%%%%

\maketitle

\section{Introduction} %%%%%%%%%%%%%%%%%%%%%%%%%%%%%%%%%%%%%%%%%%%%%%%%%%%%%%%%%%%%%%%%%%%%%%%%%%%%%%%%%%%%%%%%%%%%%%%%%%%%%%%%%%%%%%%%%%%%%%%%
%%%%%%%%%%%%%%%%%%%%%%%%%%%%%%%%%%%%%%%%%%%%%%%%%%%%%%%%%%%%%%%%%%%%%%%%%%%%%%%%%%%%%%%%%%%%%%%%%%%%%%%%%%%%%%%%%%%%%%%%%%%%%%%%%%%%%%%%%%%%%%%

The fields of science have long benefited from interdisciplinary interactions, and the application of mathematics to these fields is no exception.  With its foundation in mathematics and home in physics, the exploration of quantum mechanics has beyond any doubt cast us into the atomic era, and chief in the development of this knowledge were the independently developed models of quantum mechanics, namely Werner Heisenberg’s matrix mechanics and Erwin Schr\"odinger’s wave mechanics.  Where matrix mechanics proposed a governance of subatomic phenomena by certain interacting linear transformations, wave mechanics distinguished itself by the application of partial differential equations. Remarkably, each model yielded the same experimental predictions, suggesting a unity of theory despite cosmetic differences. Using sophisticated new tools developed by Marshall Stone (which today serve as important rudiments in the study of operator algebras), John von Neumann successfully proved the formal mathematical equivalence of these two models in what is now called the Stone--von Neumann Theorem. Loosely stated, the theorem proves that any self-adjoint representation of Heisenberg's linear operators is merely a multiple of the irreducible representation coming from solutions to Schr\"odinger's wave equation. The essential strategy involved associating to a coupled pair of linear transformations two families of linear operators indexed by the topological group $\br^n$. 

Von Neumann proceeded to study rings of operators, initiating the study of operator algebras which reaches as far in the field of mathematics as quantum mechanics does in physics. Indeed, operator algebraists have developed deep interactions with numerous ``classical'' mathematical fields, mirroring the drive to unify physics under one ``standard model''. Throughout this enrichment of the field, a steady interest in Stone and von Neumann's previous investigations remained, sustaining a perpetual refinement and generalization of ideas. Major influences include George Mackey and Marc Rieffel, with successive advancement from numerous others. Recently, Ismert and the second author provided a significant innovation, classifying representations of \cs-dynamical systems of abelian groups represented on Hilbert modules. Quigg and the first two authors expanded this classification to include dynamical systems of general (nonabelian) locally compact groups using divergent techniques from the parallel but distinct lenses of duality for \cs-actions and coactions. 

From the perspective of nonabelian duality, the divergence of technique is a necessary cost for the expansion beyond actions of an abelian group. However, recent decades have offered the appealing new theory of locally compact quantum groups to unify these parallel developments under a~common theory. In this article, we call upon this theory to integrate the work of Quigg and the first two authors, recovering numerous independent results inspired by Stone and von Neumann's century-old work, and offering a new analysis of spectral phenomena which may be regarded as a foundation for the multiplicity results which appear in the literature. Figuring prominently in this article is the emphasis on representations of locally compact quantum groups, which are a valuable tool for quantum groups which appears implicitly in the coaction setting of \cite{modular}. These objects, other conventions for locally compact quantum groups, and their crossed products are recorded in \S \ref{prelim}, together with a new result of general interest to the quantum group community: preservation of the full crossed products of universal quantum groups under duality, which implies under a regularity assumption maximality of the trivial actions. Heisenberg representations of \cs-quantum group dynamical systems are defined in \S \ref{hrep}, where we also explore close relationships with prior evolutionary expressions. Together with these relationships, Theorem \ref{covcorr} serves to confirm the proposed definition. In \S \ref{sec:main}, for a given dynamical system we introduce the  Schr\"odinger representation, and show that it is a prototype for the Heisenberg representation. For dynamical systems comprising maximal actions and elementary coefficient algebras, we deduce that every Heisenberg representation is a multiple of the Schr\"odinger representation, and with this characterization in hand, we prove our main result: given a dynamical system $(\GG, A, \alpha)$ where $\GG$ is a locally compact quantum group and $A$ admits a faithful state, if every Heisenberg representation is a multiple of the Schr\"odinger representation, then the spectrum of the \cs-algebra $\widehat{\GG}^\op\ltimes (\GG\ltimes A)$ is a one point space. This result is new already for the case of $A=\bc$.  In this setting, we conclude that (1) if $A$ is separable, then $A$ was an elementary \cs-algebra all along; and (2) if $\GG$ is a regular locally compact group, then the action is maximal. As a final corollary, we show that if $\GG$ is regular, the generalized Mackey--Stone--von Neumann Theorem holds for $(\GG, \bc, \text{triv})$ if and only if $\GG$ is strongly regular (equivalently: the trivial action of $\mathbb{G}$ on $\mathbb{C}$ is maximal).

% Finally, to further justify our assumptions, we investigate the conditions under which one expects a Stone--von Neumann-type result to hold. In our second main result, we prove that if a~\cs-quantum group dynamical system $(\GG,A,\alpha)$, where $A$ admits a faithful state, satisfies the generalized Stone--von Neumann theorem, then the spectrum of the iterated crossed product $\widehat{\GG}^{\rm op}\ltimes(\GG\ltimes A)$ consists of a single point. In this case, if one assumes that $A$ is separable, we infer that it has to be elementary. Much in the same way, if $\GG$ is regular, the action $\alpha$ has to be maximal. As a corollary,  that if the generalized Stone--von Neumann--Mackey theorem holds $\GG$ a regular locally compact quantum group, then its dual $\widehat{\GG}$ is strongly regular. Therefore, we obtain a novel result already in the case $A=\mathbb{C}$ which can be used to identify strongly regular locally compact quantum groups.

\subsection*{Acknowledgments}
The authors are deeply indebted to the referees, whose commentary vastly improved the exposition of this article. 

The work of JK was partially supported by  FWO grant 1246624N
as well as EPSRC grants EP/T03064X/1 and EP/T030992/1. 

LHall recognizes the Zuckerman Institute, who supports him as a Zuckerman Postdoctoral Scholar in the 2023-2024 cohort. 

MT was partially supported by NCN grant 2020/36/C/ST1/00082.

\subsection*{Data Availability Statement} 
Data sharing is not applicable to this article as no new data were created or analysed in this study.

\subsection*{Conflict of interest}
On behalf of all authors, the corresponding author states that there is no conflict of interest.

% PRELIMINARIES %%%%%%%%%%%%%%%%%%%%%%%%%%%%%%%%%%%%%%%%%%%%%%%%%%%%%%%%%%%%%%%%%%%%%%%%%%%%%%%%%%%%%%%%%%%%%%%%%%%%%%%%%%%%%%%%%%%%%%%%%%%%%%%

\section{Preliminaries}\label{prelim}

In this section we record the conventions we will use throughout the article. All homomorphisms are assumed to be $*$-preserving. An elementary \cs-algebra is a \cs-algebra $A$ which is isomorphic to the \cs-algebra of compact operators over some Hilbert space $H$. 

\subsection{Hilbert Modules}
We refer to \cites{lance, rw, buss} for details on Hilbert modules. Let $A$ be a \cs-algebra. A (right) \emph{Hilbert A-module} is a complex vector space $X$, denoted $X_A$ when confusion seems possible, which is also a right $A$-module equipped with an $A$-valued inner product $\langle \cdot, \cdot \rangle_A\colon  X\times X\to A$ satisfying various compatibility conditions. The inner product induces a natural norm on $X$, $\|x\| = \|\langle x,x\rangle_A \|_A^{1/2}$, and $X$ is required to be complete with respect to this norm; $X$ is called (right) \emph{full} provided that the ideal $\text{span}\{\langle y, x \rangle\,|\, y,x\in X\}$ is dense in $A$. Standard examples include any \cs-algebra $A$, which is a Hilbert $A$-module with inner product $\langle b, a\rangle_A = b^*a$; and Hilbert $\bc$-modules, which are Hilbert spaces with the convention that the ordinary inner product be conjugate-linear in the first variable. 

The set of \emph{adjointable operators} on $X$ is defined as
\[\mc{L}(X) = \{T\colon X\to X| \text{ there is } T^*\colon X\to X \text{ such that }\langle T^*\eta, \xi \rangle_A = \langle \eta, T\xi \rangle_A \text{ for all }\xi,\eta\in X\}.\] 
$\mc{L}(X)$ is a \cs-algebra which generalizes bounded Hilbert-space operators. When $X=A$ is a \cs-algebra, we denote $\mc{L}(A) = M(A)$ and call $M(A)$ the \emph{multiplier algebra} of $A$ --- it is the universal unital \cs-algebra which contains $A$ as an essential ideal. A $B-A$ \emph{correspondence} is a Hilbert $A$-module $X$ together with a nondegenerate homomorphism $B\to \mc{L}(X),$ by which we mean that $\ov{\oon{span}}\{b\cdot x|~ b\in B \text{ and } x\in X\}=X$ (or equivalently $\{b\cdot x|~ b\in B \text{ and } x\in X\}=X$ by the Cohen-Hewitt Factorization Theorem \cite{hewitt}). A \emph{$B-A$ imprimitivity bimodule} is a $B-A$ correspondence $X$ equipped with a $B$-valued inner product $\langle\cdot, \cdot \rangle_B\colon X\times X\to B$ which is left full and compatible with the $A$-valued inner product in the sense that $\langle \zeta,\eta\rangle_B\cdot \xi = \zeta\cdot \langle \eta, \xi\rangle_A$. For us, there are two imprimitivity bimodules of fundamental importance. The first is a Hilbert space $H$, which is a $\bk(H)-\bc$ imprimitivity bimodule. The left action is given by the identity map, and the left-inner-product structure is determined by the rule $\langle z, y \rangle_{\bk(H)} = \theta_{z, y}$, where $\theta _{z,y}$ is the standard rank-one operator $\theta_{z,y}(x) = z\langle y, x \rangle_\bc$ --- compatibility of the inner products is by definition. The other natural imprimitivity bimodule is given by a \cs-algebra $A$, which is an $A-A$ imprimitivity bimodule. The left action is simply left multiplication, and the left-inner-product structure is given by $\langle c, b\rangle = cb^*,$ thus $\langle c, b\rangle a = cb^*a = c\langle b, a\rangle.$ Fullness follows from an approximate-identity argument.  

In \cite{modular}, Quigg and the first two authors completely classified the correspondences over elementary \cs-algebras. In the theorem below, one correspondence $Y$ is a \emph{multiple} of another $X$ provided $Y$ is unitarily equivalent to a direct sum $\oplus_{S} X$. 

\begin{Thm}\cite{modular}*{Theorem 3.5}\label{modularsvn}
Let $A$ be an elementary \cs-algebra and $X$ an $B-A$ imprimitivity bimodule. Then every $B-A$ correspondence $Y$ is multiple of $X$.  
\end{Thm}

We assume familiarity with the spatial tensor product of \cs-algebras which will simply be denoted by $\otimes$. Given Hilbert modules $X_A$ and $Y_B$, one can define an external tensor product $X\otimes Y$, which is a Hilbert $A\otimes B$-module and contains the algebraic tensor product of $X$ and $Y$ as a dense subspace. Given $T\in \mc{L}(X)$ and $S\in \mc{L}(Y)$, then there is $T\otimes S\in \mc{L}(X\otimes Y)$ determined on elementary tensors in the obvious way. We will make extensive use of leg-numbering notation, which we briefly describe. For $T\in \mc{L}(X\otimes X),$ $T_{ij}\in \mc{L}(X\otimes X\otimes X)$ are defined for $1\leq i<j\leq 3$ by $T_{12} = T\otimes 1$, $T_{23} = 1\otimes T$, and $T_{13} = (1\otimes \Sigma)T_{12}(1\otimes \Sigma)$, where $\Sigma\colon x\otimes y\mapsto y\otimes x$ is the flip operator on $X\otimes X$. For the flip map on the tensor product of two algebras $A\otimes B$, we will use symbol $\chi$.

\subsection{The Millinery of Locally Compact Quantum Groups}

Our main references for locally compact quantum groups are \cites{kv, buss}. Other standard references include \cites{kustermansuniversal, venchilada, KVvn}. 

A \emph{locally compact quantum group} $\GG$ consists of a $C^*$-algebra $C_0(\GG)$ together with a nondegenerate $*$-homomorphism $\Delta\colon C_0(\GG)\to M(C_0(\GG)\otimes C_0(\GG))$, called comultiplication which satisfies the coassociativity $(\Delta\otimes \id)\Delta=(\id\otimes \Delta)\Delta$ and density conditions
\[\begin{split}
C_0(\GG)&=\ov{\lin}\{(\omega\otimes\id)\Delta(x) |~ \omega\in C_0(\GG)^*,x\in C_0(\GG)\}\\
&=\ov{\lin}\{(\id\otimes \omega)\Delta(x) |~ \omega\in C_0(\GG)^*,x\in C_0(\GG)\}.
\end{split}\]
Furthermore, by definition, there are two faithful, approximately KMS weights $\vp,\psi$ on $C_0(\GG)$ which are respectively left- and right-invariant 
(for details see \cite[Definition 4.1]{kv}). One can equivalently work in the von Neumann-algebraic setting 
(\cite[Definition 1.1]{KVvn}), but we will stay in the language of noncommutative topology, i.e., \cs-algebras.

Associated with the left Haar weight $\vp$ is the GNS Hilbert space, denoted  $L^2(\GG)$, and faithful representation $\pi_\vp\colon C_0(\GG)\to \bdops(L^2(\GG))$. Since $\pi_\vp$ is faithful, it is customary to consider $C_0(\GG)$ as acting on $L^2(\GG)$ and neglect writing $\pi_\vp$, and we will follow this custom. One defines a unitary $\ww^{\GG}\in \bdops(L^2(\GG)\otimes L^2(\GG))$ called the Kac-Takesaki operator. It is a multiplicative unitary, i.e., it satisfies the Pentagonal Relation
$$
\ww^{\GG}_{12}\ww^{\GG}_{13}\ww^{\GG}_{23}=\ww^{\GG}_{23}\ww^{\GG}_{12}.
$$
From $\ww^{\GG}$, one recovers the $C^*$-algebra $C_0(\GG) = \ov{\lin}\{(\id\otimes \omega)(\ww^{\GG})|~ \omega\in \bdops(L^2(\GG))_* \}$ and comultiplication $\Delta(a) = \ww^{\GG *}(1\otimes a)\ww^{\GG}\;(a\in C_0(\GG))$. One defines the \emph{dual} locally compact quantum group $\wh{\GG}$ --- its $C^*$-algebra may be recognized as $C_0(\widehat{\GG})=\ov{\lin}\{(\omega\otimes\id)(\ww^{\GG})|~\omega\in \bdops(L^2(\GG))_*\}$ and the comultiplication is given by $\widehat{\Delta}(a)= \ww^{\wh\GG *}(1\otimes a)\ww^{\wh\GG}\;(a\in C_0(\wh{\GG}))$, where $\ww^{\wh\GG} = \chi ( \ww^{\GG *})$. In fact, $\ww^{\wh\GG}$ is the Kac-Takesaki operator of $\wh{\GG}$. The remaining features of $\widehat{\GG}$ will be likewise decorated with hats. One can prove that in fact $\ww^{\GG}\in M(C_0(\GG)\otimes C_0(\wh\GG))$ and
\[
(\Delta\otimes \id)\ww^{\GG}=\ww^{\GG}_{13}\ww^{\GG}_{23},\qquad
(\id\otimes \wh\Delta)\ww^{\GG}=
\ww^{\GG}_{13} \ww^{\GG}_{12}.
\]
The culmination of the theory is the recovery of generalized Pontryagin duality $\widehat{\widehat{\GG}}\cong\GG$.
 
Every locally compact group $G$ gives rise to two quantum groups. The first one is $\GG$, with associated \cs-algebra $C_0(\GG)=C_0(G)$, comultiplication acting by $\Delta(f)(x,y)=f(xy)$ and weights $\vp,\psi$ given by integration with respect to left (resp.~right) Haar measure. One typically identifies $\GG$ with $G$. The second one, the dual quantum group, is described by objects studied in abstract harmonic analysis: $C_0(\wh{G})=C_r^*(G)$ is the reduced group \cs-algebra and comultiplication acts on generators via $\wh{\Delta}(\lambda(g))=\lambda(g)\otimes \lambda(g)$. Both Haar weights of $\wh{\GG}$ are equal to the Plancherel weight. Nonclassical analogs include compact quantum groups (which entails $C_0(\GG)$ possess a unit) and discrete quantum groups (where $C_0(\GG)$ is a direct sum of matrix algebras). 

Out of a quantum group $\GG$, one builds yet another \cs-algebra, $C_0^u(\GG)$, which should be thought of as a universal version of $C_0(\GG)$. It comes with its own comultiplication: a non-degenerate $*$-homomorphism $\Delta^u\colon C_0^u(\GG)\rightarrow M(C_0^u(\GG)\otimes C_0^u(\GG))$ and the \emph{reducing map}, which is a $*$-epimorphism $\lambda\colon C^u_0(\GG)\rightarrow C_0(\GG)$ respecting comultiplications (i.e. $\Delta\circ \lambda = (\lambda\otimes \lambda) \circ \Delta^u$). Furthermore, it has a very useful universal property concerning representations of $\wh\GG$ which we will encounter in section \ref{sec:corep} (see also \cite{kustermansuniversal}). 

Let $J$ be the modular conjugation associated to the left Haar weight. The \emph{commutant of $\GG$} is a locally compact quantum group $\Comm{\GG}$ with \cs-algebra $C_0(\Comm{\GG})=J C_0(\GG)J$ and comultiplication $\Comm{\Delta}(a)=(J\otimes J) \Delta(JaJ) (J\otimes J)$ for $a\in C_0(\Comm{\GG})$. Another locally compact quantum group associated with $\GG$ is the \emph{opposite} of $\GG$ with $C_0(\GG^{\op})=C_0(\GG)$ and $\Delta^\op = \chi\circ \Delta$. This millinery of examples are related by various relations, namely 
\[\widehat{\GG^\op}= \widehat{\GG}^\c \qquad \widehat{\Comm{\GG}}=\widehat{\GG}^\op \qquad \GG^{\op, \c}=\GG^{\c,\op}\cong_{\widehat{J}J} \GG;\]
The final relation states that the opposite of the commutant of $\GG$ is unitarily equivalent to $\GG$ by the unitary $\widehat{J}J$ (see \cite[Section 4]{KVvn}). Each of the quantum groups we consider above admits universal companions, and we will decorate the respective reducing maps using the same notation, e.g. $\hat{\lambda}\colon C_0^u(\widehat{\GG})\to C_0(\widehat{\GG})$, and similarly with $\lambda^\op$ and $\lambda^c$. 

The last definition introduced here is the notion of regularity. By definition, a locally compact quantum group $\GG$ is \emph{regular} if $\ov{\oon{span}}\{ (\id\otimes \omega)(\Sigma \ww^{\GG}) |~ \omega\in \bdops(L^2(\GG))_*\}=\bk(L^2(\GG))$. One can show that $\GG$ is regular if and only if $\wh\GG$ is regular. Furthermore, all compact or discrete quantum groups are regular as well as quantum groups with trivial scaling group (see \cites{BaajSkandalis,venchilada} and references therein).

\subsection{Representations} \label{sec:corep}

Given a locally compact quantum group $\GG$ and a Hilbert $A$-module $X$, a~\emph{left representation} of $\GG$ (or a~{\em left corepresentation} of $C_0(\mathbb{G})$) on $X$ is a unitary $u\in \mc{L}(C_0(\GG)\otimes X)$ which satisfies $(\Delta\otimes \id)(u) = u_{13}u_{23}$. Likewise, a \emph{right representation} is a unitary $v\in \mc{L}(X\otimes C_0(\GG))$ which satisfies $(\id\otimes \Delta)(v) = v_{12}v_{13}$. The assignment $u\mapsto u^\op = \chi(u)^*$ determines a bijective correspondence between 
left representations of $\GG$ and right representations of $\GG^\op$. The Kac-Takesaki operator is a distinguished left representation of $\GG$ on $L^2(\GG)$. In fact, $\ww^{\GG}\in M(C_0(\GG)\otimes C_0(\widehat{\GG}))$, and 
regarded this way we call $\ww^{\GG}$ the \emph{left regular representation} of $\GG$ (on $C_0(\widehat{\GG})$). 
We also have the right regular representation of $\GG$ (which is indeed a right representation corresponding to a~particular left representation of $\GG^\op$), 
which is a unitary $\vv^{\GG}\in M(C_0(\widehat{\GG}^\c)\otimes C_0(\GG))$ satisfying $(\id\otimes \Delta)(\vv^{\GG}) = \vv^{\GG}_{12}\vv^{\GG}_{13}$. It also implements the comultiplication, namely $\Delta(x)=\vv^{\GG} (x\otimes 1) \vv^{\GG *}$ for $x\in C_0(\GG)$.

There is a universal (left) representation of $\GG$, written 
${\wW}^{\GG}\in M(C_0(\GG)\otimes C_0^u(\widehat{\GG}))$ and satisfying the relation 
$(\Delta\otimes\id)({\wW}^{\GG})={\wW}^{\GG}_{13}{\wW}^{\GG}_{23}$. 
The comultiplication $\widehat{\Delta}_u$ on $C_0^u(\widehat{\GG})$ is defined exactly so that $(\id\otimes \widehat{\Delta}_u)({\wW}^{\GG}) = {\wW}^{\GG}_{13}{\wW}^{\GG}_{12}$. The right leg of ${\wW}^{\GG}$ generates $C_0^u(\wh\GG)$ in the sense that
$C_0^u(\wh\GG)=\ov{\lin} \{ (\omega\otimes \id){\wW}^{\GG}\,|\, \omega\in \bdops( L^2(\GG))_* \}$. Moreover, the universal dual $C_0^u(\widehat{\GG})$ encodes the representation theory of $\GG$ in the sense that any left representation $u\in \mc{L}(C_0(\GG)\otimes X)$ corresponds to exactly one nondegenerate representation $\mu\colon C_0^u(\widehat{\GG})\to \mc{L}(X)$, characterized by $(\id\otimes \mu)({\wW}^{\GG}) = u$ (\cite[Proposition 2.14]{kustermansuniversal}). Similar statements hold for the universal (right) representation of $\GG$, $\mathds{V}^{\GG}\in M(C_0^u(\widehat{\GG^\op})\otimes C_0(\GG))=M(C_0^u(\widehat{\GG}^\c)\otimes C_0(\GG))$. All of this applies equally well to $\widehat{\GG}$, and we will in particular encounter the right universal representation of $\widehat{\GG}$, written $\mathds{V}^{\widehat{\GG}}\in M(C_0^u(\GG^c)\otimes C_0(\widehat{\GG}))$.\\

\subsection{Crossed Products} 

Let $A$ be a \cs-algebra and $\GG$ a locally compact quantum group. A \emph{left action} of $\GG$ on $A$  is a nondegenerate homomorphism $\alpha\colon A\to M(C_0(\GG)\otimes A)$ which satisfies the following two conditions 
\begin{align*}
 (\id\otimes\alpha)\circ \alpha = (\Delta\otimes \id)\circ \alpha\quad \text{ and }\quad
\overline{\text{span}}\,(C_0(\GG) \otimes 1 )\alpha(A) = C_0(\GG)\otimes A.
 \end{align*}
 The second condition is variously referred to as \emph{continuity, coaction nondegeneracy,} or the \emph{Podleś condition} in the literature. The data $(\GG, A, \alpha)$ is also called %a \emph{left coaction}, or otherwise 
a \emph{left dynamical system}. One analogously defines a \emph{right action} of $\GG$ on $A$ by swapping the positions of $A$ and $\GG$ in the codomain of $\alpha$; as above, left actions of $\GG$ on $A$ correspond to right actions of $\GG^\op$ on $A$. The triple $(A,\GG,\alpha)$ denotes a right dynamical system. Notice that the comultiplication of a locally compact quantum group is an example of a left (and right) action of $\GG$ on itself. Another example is given by the \emph{(left) trivial action} of $\GG$ on $A$, denoted $(\GG, A, \text{triv})$. Here, $\text{triv}(a) = 1\otimes a$.

 %\lucas{Are there any other important coactions we should include here?} 

% \begin{remark}
%In the classical setting of locally compact groups, right actions appear as the standard choice. In the quantum setting, however, left representations and consequently left actions appear to be more natural. Of course, in view of the comments above, the two notions are really more or less equivalent --- the burden of working with right actions is purely notational. Nevertheless, in this article, we will have use for both left and right actions. 
% \end{remark}

Given a left dynamical system $(\GG, A ,\alpha)$ and any Hilbert $B$-module $X = X_B$, a \emph{covariant representation} of $(\GG, A, \alpha)$ on $X_B$ is a pair of nondegenerate homomorphisms $(\mu,\pi)$, with $\mu\colon C_0^u(\widehat{\GG})\to \mc{L}(X_B)$ and $\pi\colon A\to \mc{L}(X_B)$, subject to the condition that 
\[(\id \otimes \pi)\circ\alpha(a) = (\id \otimes \mu)(\wW^\GG)^*(1 \otimes \pi(a)) (\id\otimes \mu)(\wW^\GG)\]
for all $a\in A$. Notice that the right hand side is conjugation by the left unitary representation $(\id\otimes \mu)(\wW^\GG)$. Elsewhere in the literature, a covariant representation is regarded as a pair comprising a homomorphism of the coefficient algebra and a left representation on the common Hilbert module $X$. Given a right dynamical system $(A, \GG, \alpha)$, a covariant representation on $X=X_B$ is a~pair $(\pi, m)$ with $\pi\colon A\to \mc{L}(X_B)$ and $m\colon C_0^u(\widehat{\GG}^\c)\to \mc{L}(X_B)$ subject to 
\[(\pi\otimes \id)\circ \alpha(a) = (m\otimes \id)(\mathds{V}^\GG)(\pi(a)\otimes 1)(m\otimes \id)(\mathds{V}^\GG)^*\]
An important example of a covariant representation for $(\GG, A, \alpha)$ is given on the Hilbert module $L^2(\GG)\otimes A$ by the pair $( \hat{\lambda} \otimes 1, \alpha)$ -- it is called the \emph{the regular representation} of the dynamical system.

For the left dynamical system $(\GG, A, \alpha)$, the \emph{full crossed product} is a \cs-algebra $\GG\ltimes A$ together with a covariant representation 
 \[
j_\GG\colon C_0^u(\widehat{\GG})\to M(\GG\ltimes A),\qquad j_A\colon A\to M(\GG\ltimes A) 
 \]
 which is universal in the following sense: whenever $(\mu,\pi)$ is a covariant representation of $(\GG, A, \alpha)$ on $X$, there is a unique nondegenerate homomorphism $\mu\ltimes \pi\colon  \GG\ltimes A\to \mc{L}(X)$, called the \emph{integrated form} (of the covariant representation), so that 
 \[\mu=(\mu\ltimes \pi)\circ j_\GG,\qquad
 \pi=(\mu\ltimes \pi)\circ j_A.
 \]
 On account of the above identifications, one may also \emph{disintegrate} an integrated form $\mu\ltimes \pi$ to recover the original covariant representation by precomposing by the universal covariant representations. The full crossed product exists, is unique up to isomorphism, and moreover 
 (see \cite{venchilada,vergnioux})
 \begin{equation}\label{crospan}
\GG\ltimes A=\overline{\rm span}\{j_A(a)j_\GG(h)~|~a\in A,h\in C_0^u(\widehat{\GG})\}.
\end{equation}
For example, the crossed product for the trivial action $(\GG, A, \text{triv})$ is the maximal tensor product $C_0^u(\wh{\GG})\otimes_{\textnormal{max}} A$. For other fundamental examples, one constructs the full crossed product $\GG\ltimes C_0(\GG)$ for the dynamical system $(\GG, C_0(\GG), \Delta_\GG)$, and the reduced crossed product $\GG\ltimes_r C_0(\GG) = (\hat{\lambda}\otimes 1\ltimes \Delta_{\GG}) (\GG\ltimes C_0(\GG))$. A locally compact quantum group $\GG$ is regular if and only if $\GG\ltimes_r C_0(\GG)\cong \bk(L^2(\GG))$. Similarly, $\GG$ is called \emph{strongly regular} provided that $\GG\ltimes C_0(\GG)\cong \bk(L^2(\GG))$ (\cite[Definition 2.11]{venchilada}). We remark that we obtain the same property, if we assume that these isomorphisms are implemented by canonical maps, or just abstract.

Let $(\GG,A,\alpha)$ be again an arbitrary left dynamical system. There is a left action $\wh{\alpha}^{\rm op}$ of $\wh{\GG}^{\rm op}$ on $\GG\ltimes A$, called the \emph{dual action to $\alpha$}, uniquely characterized by
\begin{equation}\label{dc}
\wh{\alpha}^{\rm op}(j_\GG(h))=(\hat{\lambda}\otimes j_\GG)\Delta^u_{\wh{\GG}^{\rm op}}(h),\quad h\in C^u_0(\wh{\GG}),
\qquad 
\wh{\alpha}^{\rm op}(j_A(a))=1\otimes j_A(a),\quad a\in A.
\end{equation}
Indeed, one may check that $((\hat{\lambda}\otimes j_{\GG})\Delta_{\wh{\GG}^{\rm op}}^u,1\otimes j_A)$ forms a covariant representation of $(\GG,A,\alpha)$ on $C_0(\wh{\GG}^{\rm op})\otimes (\GG\ltimes A)$. 
%Clearly both $(\lambda_{\wh{\GG}}\otimes j_{\GG})\Delta_{\wh{\GG}^{\rm op}}^u$ and $1\otimes j_A$ are non-degenerate $*$-homomorphisms from $C_0^u(\wh{\GG}^{\rm op})$ (resp.~$A$) to $M\bigl( C_0(\wh{\GG}^{\rm op})\otimes (\GG\ltimes A)\bigr)$. Take $a\in A$. We have
% \[\begin{split}
% &\quad\;
% (\id\otimes (\lambda_{\wh\GG}\otimes j_{\GG})\Delta_{\wh{\GG}^{\rm op}}^u)(\wW^{\GG})^* 
% (1\otimes 1\otimes j_A(a))
% (\id\otimes (\lambda_{\wh\GG}\otimes j_{\GG})\Delta_{\wh{\GG}^{\rm op}}^u)(\wW^{\GG})\\
% &=
% ( W^{\GG}_{12}(\id\otimes j_{\GG})(\wW^{\GG})_{13} )^* (1\otimes 1\otimes j_A(a))
% (  W^{\GG}_{12}(\id\otimes j_{\GG})(\wW^{\GG})_{13})\\
% &=
% (\id\otimes j_{\GG})(\wW^{\GG})^*_{13}
% (1\otimes 1\otimes j_A(a))
% (\id\otimes j_{\GG})(\wW^{\GG})_{13}\\
% &=
% (\id\otimes j_A)\alpha(a)_{13}=
% (\id\otimes (1\otimes j_A))\alpha(a),
% \end{split}\]
Then, by the universal property of $\GG\ltimes A$, we obtain a non-degenerate $*$-homomorphism $\wh{\alpha}^{\rm op}=(\hat{\lambda}\otimes j_{\GG})\Delta^u_{\wh{\GG}^{\rm op}}\ltimes (1\otimes j_A)\colon \GG\ltimes A\rightarrow M(C_0(\wh{\GG}^{\rm op})\otimes (\GG\ltimes A))$ acting via \eqref{dc}.

For a dynamical system $(\GG, A, \alpha)$ with $\GG$ regular, there is a canonical surjection $\wh{\GG}^{\rm op}\ltimes (\GG\ltimes A)\rightarrow \bk(L^2(\GG))\otimes A,$ which is given by the integrated form of the covariant representation $(\lambda^c\otimes 1, (\hat{\lambda}\otimes 1) \ltimes \alpha)$; the action $\alpha$ is called \emph{maximal} provided this surjection is an isomorphism~\cite[Definition~2.13]{venchilada}. 

%\begin{remark}
We are now prepared to present a result of general interest to the field of locally compact quantum groups. 
\begin{Prop}\label{maxtrivdual}
Let $\GG$ be a locally compact quantum group. We have the following isomorphisms\footnote{We are grateful to Stefaan Vaes for pointing us towards the proof of isomorphism $\GG\ltimes C_0^u(\GG)\cong \GG\ltimes C_0(\GG)$.}:
\begin{enumerate}
\item $\GG\ltimes C_0^u(\GG)\cong \widehat{\GG}\ltimes C_0^u(\widehat{\GG})$,
\item $ \GG^{\op}\ltimes C_0^u(\GG^{\op})\cong (\GG\ltimes C_0^u(\GG))^{\op},$
\item $\GG\ltimes C_0^u(\GG)\cong \GG\ltimes C_0(\GG)$. 
\end{enumerate}
Assume that $\GG$ (equivalently: $\GG^{\op}$ or $\wh{\GG}$) is regular. As a consequence we conclude that the following properties are equivalent:
\begin{itemize}
\item $\GG$ is strongly regular,
\item $\GG^{\op}$ is strongly regular,
\item $\wh{\GG}$ is strongly regular,
\item trivial action of $\GG$ on $\mathbb{C}$ is maximal.
\end{itemize}
\end{Prop}
Let us remark that each of the above isomorphisms is implemented by natural maps.

\begin{proof}
Let $j_{\GG}$ and $j_{C_0^u(\GG)}$ denote the canonical maps for the full crossed product $\GG\ltimes C_0^u(\GG)$, and let $j_{\widehat{\GG}}$ and $j_{C_0^u(\widehat{\GG})}$ be the analogous maps for $\widehat{\GG}\ltimes C_0^u(\widehat{\GG})$. We claim that $(j_{C_0^u(\widehat{\GG})},j_{\widehat{\GG}})$ is a covariant representation of the dynamical system $(\GG, C_0^u(\GG), (\lambda\otimes \id)\Delta^u)$ on $M(\widehat{\GG}\ltimes C_0^u(\widehat{\GG}) )$. Indeed, since $(j_{\widehat{\GG}},j_{C_0^u(\widehat{\GG})})$ forms a covariant representation of $(\widehat{\GG},C_0^u(\widehat{\GG}),(\widehat{\lambda}\otimes{\rm id})\widehat{\Delta}^u)$, for $x^u=(\omega\otimes\id)(\wW^{\widehat{\GG}})$, $\omega\in L^1(\widehat{\GG})$, we have that
\[\begin{split}
&\quad\;
(\id\otimes j_{C_0^u(\widehat{\GG})})(\wW^{\GG})^*
(1\otimes j_{\widehat{\GG}}(x^u))
(\id\otimes j_{C_0^u(\widehat{\GG})})(\wW^{\GG})\\
&=
(\omega\otimes\id\otimes\id)\bigl(
(\id\otimes j_{C_0^u(\widehat{\GG})})(\wW^{\GG})^*_{23}
(\id\otimes j_{\widehat{\GG}})(\wW^{\widehat{\GG}})_{13}
(\id\otimes j_{C_0^u(\widehat{\GG})})(\wW^{\GG})_{23}\bigr)\\
&=
(\id\otimes \omega\otimes\id)\bigl(
(\id\otimes j_{\widehat{\GG}})(\wW^{\widehat{\GG}})_{23}(\id\otimes j_{\widehat{\GG}})(\wW^{\widehat{\GG}})_{23}^*
(\id\otimes j_{C_0^u(\widehat{\GG})})(\wW^{\GG})^*_{13}
(\id\otimes j_{\widehat{\GG}})(\wW^{\widehat{\GG}})_{23}\\
&\quad\quad\quad\quad\quad\quad
\quad\quad\quad\quad\quad\quad
\quad\quad\quad\quad\quad\quad
\quad\quad\quad\quad\quad\quad
\quad\quad\quad\quad\quad\quad
(\id\otimes j_{C_0^u(\widehat{\GG})})(\wW^{\GG})_{13}\bigr)\\
&=
(\id\otimes \omega\otimes\id)\bigl(
(\id\otimes j_{\widehat{\GG}})(\wW^{\widehat{\GG}})_{23}
(\id\otimes\id\otimes j_{C_0^u(\widehat{\GG})})
(\id\otimes (\widehat{\lambda}\otimes \id)\widehat{\Delta}^u)(\wW^{\GG *})
(\id\otimes j_{C_0^u(\widehat{\GG})})(\wW^{\GG})_{13}\bigr)\\
&=
(\id\otimes \omega\otimes\id)\bigl(
(\id\otimes j_{\widehat{\GG}})(\wW^{\widehat{\GG}})_{23}
\ww^{\GG *}_{12}
(\id\otimes j_{C_0^u(\widehat{\GG})})(\wW^{\GG *})_{13}
(\id\otimes j_{C_0^u(\widehat{\GG})})(\wW^{\GG})_{13}\bigr)\\
&=
(\id\otimes \omega\otimes\id)\bigl(
(\id\otimes j_{\widehat{\GG}})(\wW^{\widehat{\GG}})_{23}
\ww^{\GG *}_{12}\bigr)=
(\omega\otimes \id\otimes\id)\bigl(
(\id\otimes j_{\widehat{\GG}})(\wW^{\widehat{\GG}})_{13}
\ww^{\widehat{\GG}}_{12}\bigr)\\
&=
(\omega\otimes \id\otimes\id)\bigl(
(\id\otimes \id\otimes j_{\widehat{\GG}})
(\id\otimes (\lambda\otimes\id)\Delta^u)(\wW^{\widehat{\GG}})\bigr)=
( \id\otimes j_{\widehat{\GG}})
((\lambda\otimes\id)\Delta^u(x^u)),
\end{split}\]
which, by continuity of the maps involved, proves the claim. Hence, there is a unique non-degenerate $*$-homomorphism $j_{C_0^u(\widehat{\GG})}\ltimes j_{\widehat{\GG}}$: $\GG\ltimes C_0^u(\GG)\rightarrow M( \widehat{\GG}\ltimes C_0^u(\widehat{\GG}))$ such that
\begin{equation}\label{eq101}
(j_{C_0^u(\widehat{\GG})}\ltimes j_{\widehat{\GG}})j_{\GG}=
j_{C_0^u(\widehat{\GG})},\quad
(j_{C_0^u(\widehat{\GG})}\ltimes j_{\widehat{\GG}})j_{C_0^u(\GG)}=
j_{\widehat{\GG}}.
\end{equation}
By duality, we obtain a non-degenerate $*$-homomorphism $j_{C_0^u(\GG)}\ltimes j_{\GG}\colon \widehat{\GG}\ltimes C_0^u(\widehat{\GG})\rightarrow M( \GG\ltimes C_0^u(\GG))$ such that
\begin{equation}\label{eq102}
(j_{C_0^u(\GG)}\ltimes j_{\GG})j_{\widehat{\GG}}=
j_{C_0^u(\GG)},\quad
(j_{C_0^u(\GG)}\ltimes j_{\GG})j_{C_0^u(\widehat{\GG})}=
j_{\GG}.
\end{equation}
Putting together \eqref{eq101} and \eqref{eq102}, we get
\[
(j_{C_0^u(\GG)}\ltimes j_{\GG})
(j_{C_0^u(\widehat{\GG})}\ltimes j_{\widehat{\GG}})j_{\GG}=
(j_{C_0^u(\GG)}\ltimes j_{\GG})j_{C_0^u(\widehat{\GG})}=j_{\GG}
\]
and
\[
(j_{C_0^u(\GG)}\ltimes j_{\GG})
(j_{C_0^u(\widehat{\GG})}\ltimes j_{\widehat{\GG}})j_{C_0^u(\GG)}=
(j_{C_0^u(\GG)}\ltimes j_{\GG})j_{\widehat{\GG}}=
j_{C_0^u(\GG)},
\]
which shows that $j_{C_0^u(\widehat{\GG})}\ltimes j_{\widehat{\GG}}$ is an isomorphism with inverse $j_{C_0^u(\GG)}\ltimes j_{\GG}$. This shows the first claimed isomorphim.

Isomorphism $(2)$ can be proved by a similar technique as above, combining the canonical maps into crossed product, universal property of full crossed product together with canonical antimultiplicative maps: unitary antipode, $C_0^u(\HH)\rightarrow C_0^u(\HH^{c})$ and  $A\rightarrow A^{\op}$ (for LCQGs $\HH$ and $C^*$-algebras $A$).

To see isomorphism $(3)$, let us first make an auxiliary observation. Consider arbitrary representation (i.e.~a non-degenerate $*$-homomorphism) $\GG\ltimes C_0^u(\GG)\rightarrow \bb(H)$; it is given by integration $\mu\ltimes \pi$ of a covariant representation $(\mu,\pi)$ of $(\GG,C_0^u(\GG),(\lambda\otimes\id)\Delta^u)$ on $H$. The covariance condition reads
\begin{equation}\label{eq5}
(\id\otimes \mu)(\wW^{\GG})^* (1\otimes \pi(y^u))
(\id\otimes \mu)(\wW^{\GG})=
(\id\otimes \pi)\circ (\lambda\otimes \id)\Delta^u(y^u)
\end{equation}
for $y^u\in C_0^u(\GG)$. In particular, for $y^u=(\id\otimes\wh{\omega})\Ww^{\GG}$ with $\wh{\omega}\in L^1(\wh{\GG})$ we obtain
\[\begin{split}
&\quad\;
(\id\otimes \mu)(\wW^{\GG})^*
(1\otimes \pi(y^u))
(\id\otimes \mu)(\wW^{\GG})=
(\id\otimes\id\otimes\wh\omega)\bigl(
(\lambda\otimes\pi\otimes\id)(\Delta^u\otimes \id)(\Ww^{\GG})\bigr)\\
&=
(\id\otimes\id\otimes\wh\omega)\bigl(
(\id\otimes\pi\otimes\id)
(
{\vV}^{\GG}_{12} \ww^{\GG}_{13} \vV^{\GG *}_{12}
)\bigr)=
(\id\otimes\id\otimes\wh\omega)\bigl(
(\id\otimes\pi)(\vV^{\GG})_{12} \ww^{\GG}_{13} 
(\id\otimes\pi)(\vV^{\GG })^*_{12}\bigr)\\
&=
(\id\otimes\pi)(\vV^{\GG})
((\id\otimes\wh{\omega})(\ww^{\GG})\otimes 1)
(\id\otimes\pi)(\vV^{\GG })^*=
(\id\otimes\pi)(\vV^{\GG})
(\lambda(y^u)\otimes 1)
(\id\otimes\pi)(\vV^{\GG })^*.
\end{split}\]
By continuity, the above equation holds for all $y^u\in C_0^u(\GG)$. Consequently, we see that $\pi$ vanishes on the kernel of $\lambda$ and there is a non-degenerate $*$-homomorphism $\widetilde{\pi}\colon C_0(\GG)\rightarrow \bb(H)$ such that $\pi=\widetilde{\pi}\circ\lambda$. One immediately sees from \eqref{eq5} that $(\mu,\widetilde{\pi})$ forms a covariant representation of $(\GG,C_0(\GG),\Delta)$, hence we obtain non-degenerate $*$-homomorphism $\mu\ltimes \widetilde{\pi}\colon \GG\ltimes C_0(\GG)\rightarrow \bb(H)$ such that $(\mu\ltimes \widetilde{\pi})(j_{\GG}^r\ltimes j_{C_0(\GG)}\lambda)=\mu\ltimes \pi$, where $j_{\GG}^r,j_{C_0(\GG)}$ are the canonical maps for the crossed product $\GG\ltimes C_0(\GG)$.

We can apply this reasoning to $(\mu,\pi)=(j_{\GG},j_{C_0^u(\GG)})$. Then we obtain map $j_{\GG}\ltimes \widetilde{j}_{C_0^u(\GG)}\colon\GG\ltimes C_0(\GG)\rightarrow \GG\ltimes C_0^u(\GG)$ satisfying
\[
(j_{\GG}\ltimes \widetilde{j}_{C_0^u(\GG)})(j_{\GG}^r\ltimes j_{C_0(\GG)}\lambda)=j_{\GG}\ltimes j_{C_0^u(\GG)}=\id.
\]
This prove that the (clearly surjective) non-degenerate $*$-homomorphism $j_{\GG}^r\ltimes j_{C_0(\GG)}\lambda\colon \GG\ltimes C_0^u(\GG)\rightarrow \GG\ltimes C_0(\GG)$ is in fact an isomorphism.

Now, let $\GG$ be a regular locally compact quantum group. Assume that the trivial action $\bc\rightarrow M(C_0(\GG)\otimes \bc)$ is maximal, i.e.~that the canonical surjection $\widehat{\GG}^{\rm op}\ltimes (\GG\ltimes \bc)\rightarrow \bk(L^2(\GG))$ is an isomorphism. Using isomorphisms proved above we have the following:
\[\begin{split}
\bk(L^2(\GG))&\cong \widehat{\GG}^{\rm op}\ltimes (\GG\ltimes \bc)=
\widehat{\GG}^{\rm op}\ltimes C_0^u(\widehat{\GG})\cong
\widehat{\GG}^{\op}\ltimes C_0^u(\widehat{\GG}^{\op})\cong
(\GG\ltimes C_0^u(\GG))^{\op}\cong
(\GG\ltimes C_0(\GG))^{\op}.
\end{split}\]
Since $\bk(L^2(\GG))^{\op}\cong \bk(L^2(\GG))$, it shows that $\GG$ is strongly regular. This reasoning can easily be reversed, hence we see that if $\GG$ is strongly regular, then $\GG\curvearrowright \bc$ is maximal. Equivalence of these properties with strong regularity of $\wh{\GG}$ follows from isomorphism $(1)$, and with strong regularity of $\GG^{\op}$ from $(2)$.
\end{proof}

\subsection{The Mackey--Stone--von Neumann Theorem for Quantum Groups}
Various authors call the Stone-von Neumann theorem the isomorphism $G\ltimes C_0(\wh{G})\cong\mathbb{K}(L^2(G))$ (see~\cite[Theorem~C.34]{raeburnwilliams} and~\cite[Theorem~3.4]{huneufangruan} for a more general result in the quantum group setting) which is a crucial ingredient in Rieffel's proof of Mackey's version of the theorem. As far as the authors know, there is no Mackey--Stone--von Neumann theorem for quantum groups. In what follows, we formulate and prove a more general result relevant to Hilbert modules over any elementary \cs-algebra. Here, with the technology described above in hand, we discuss known related results. In particular, we highlight some parallels between our work, \cite{RT18}, and \cite{Rieffel}. The authors of the former article consider quantum group-twisted tensor product \cs-algebras in the settings of maximal tensor products, building on the work in the minimal setting of \cite{mrw}. Their gadget subsumes maximal tensor products, skew-commutative tensor products of $\bZ/2\bZ$-graded \cs-algebras, and the construction of importance to this discussion, universal crossed products. The details we need are found in \cite{RT18}*{Section 6.2}. 

Consider the dynamical system $(\widehat{\GG}^\op, C_0^u(\widehat{\GG}), (\hat{\lambda}\otimes \id)\circ\Delta^u_{\widehat{\GG}^\op})$ and a covariant representation $(\mu, \pi)$ for it on a Hilbert $\bc$-module $X$ (a Hilbert space representation). In the proof of \cite{RT18}*{Theorem 6.2}, the authors show (up to chirality of left-right representations) that this is equivalent to the pair furnishing a $\ww^\GG$-commutative representation, which entails commutation of the two representations modulo some interference with respect to $\ww^\GG$; c.f. our Definition~\ref{heirep} and Proposition~\ref{covcorr}. Now, if the trivial action of $\GG$ on $\bc$ is maximal (so that $\wh{\GG}^{\op}\ltimes C_0^u(\wh{\GG})\cong \bk(L^2(\GG))$), 
%then $\widehat{\GG}$ is strongly regular, and 
then the integrated form $\mu\ltimes\pi$ implements a representation of the elementary \cs-algebra $\bk(L^2(\GG))$. Since all irreducible representations of $\bk(L^2(\GG))$ are unitarily equivalent to the identity representation, it follows that $\mu\ltimes\pi$ is a multiple of the identity representation. Upon disintegration of any irreducible subrepresentation, one may reinterpret the covariant subrepresentation as a $\ww^\GG$-commutative representation of it's own. This is essentially the strategy Rieffel used in \cite{Rieffel} to prove his version of the Mackey-Stone-von Neumann Theorem, and the principle behind our Theorem \ref{main}. Thus, maximality of the trivial action is a necessary condition. In Corollary~\ref{fullconverse}, we show that it is also sufficient.

%, which is relevant to Hilbert modules over any elementary \cs-algebra.
%Fix a quantum group $\GG$, suppose $X$ is a Hilbert $\bc$-module (a Hilbert space),and that we have both a left representation $u\in \mc{L}(C_0(\GG)\otimes X)$ of $\GG$ and a right representation $v\in\mc{L}(C_0(\widehat{\GG})\otimes X)$. 

%\end{remark}

%Let us end this section with the following definition\footnote{{\color{violet}Do we want to describe this homomorphism?}}: Then a left coaction $\alpha$ is said to be \emph{maximal} if the canonical $\star$-epimorphism $\wh{\GG}^{\rm op}\ltimes (\GG\ltimes A)\rightarrow \mc{K}(L^2(\GG))\otimes A$ is an isomorphism.

\section{Heisenberg Representations}\label{hrep}

Modular representations for group actions on operator algebras already appear in \cite{HI} and \cite{modular}, emphasizing pointwise notions of covariant representations. Representations of actions lack such a pointwise notions, and require a distinct but parallel theory, which is more appropriate from the lens of locally compact quantum groups. Below, we offer a definition of Heisenberg representations which extends the acceptable dynamics and unifies the classical setting under the common theory of representations. 

\begin{Def}\label{heirep}
Let $(\mathbb{G},A,\alpha)$ be a left dynamical system. A {\em Heisenberg representation} $(X,\rho^c,\widehat{\rho},\pi)$ of $(\GG,A,\alpha)$ consists of a Hilbert $A$-module $X$ and non-degenerate $*$-homomorphisms
\[
\rho^c\colon C_0^u(\mathbb{G}^c)\longrightarrow\mathcal{L}(X),\qquad \widehat{\rho}\colon C_0^u(\widehat{\mathbb{G}})\longrightarrow \mathcal{L}(X),\qquad \pi\colon A\longrightarrow \mathcal{L}(X),
\]
such that
\begin{enumerate}
\item[(R1)] $(\wh{\rho},\pi)$ is a covariant representation of $(\GG,A,\alpha)$ on $X$,
\item[(R2)] $\ww_{13}^\GG({\rm id}\otimes\widehat{\rho}\,)(\wW^\GG)_{12}(\rho^c\otimes{\rm id})(\mathds{V}^{\widehat{\GG}})_{23}=(\rho^c\otimes{\rm id})(\mathds{V}^{\widehat{\GG}})_{23}({\rm id}\otimes\widehat{\rho}\,)(\wW^\mathbb{G})_{12}$,
\item[(R3)] The images of $\pi$ and $\rho^c$ commute.
\end{enumerate}
\end{Def}

One certainly wishes to translate between the definition offered above and prior literary iterants. The next two propositions accomplish this, followed by a brief investigation of the relationships with this article's predecessors. 

\begin{Prop}
Let $\GG$ be a locally compact quantum group, let $X$ be a right Hilbert module over a C*-algebra $A$, and let
\[
\rho^c\colon C_0^u(\mathbb{G}^c)\longrightarrow\mathcal{L}(X),\qquad \widehat{\rho}\colon C_0^u(\widehat{\mathbb{G}})\longrightarrow \mathcal{L}(X),
\]
be non-degenerate $*$-homomorphisms. The following are equivalent:
\begin{enumerate}
\item[{\rm(R2)}] $\ww_{13}^\GG({\rm id}\otimes\widehat{\rho}\,)(\wW^\GG)_{12}(\rho^c\otimes{\rm id})(\mathds{V}^{\widehat{\GG}})_{23}=(\rho^c\otimes{\rm id})(\mathds{V}^{\widehat{\GG}})_{23}({\rm id}\otimes\widehat{\rho}\,)(\wW^\mathbb{G})_{12}$.
\item[{\rm(R2')}] $(\widehat{\rho},\rho^c)$ is a covariant representation of %the right dynamical system 
$(C^u_0(\widehat{\GG}),\widehat{\GG},({\rm id}\otimes\hat{\lambda})\Delta_{\wh{\GG}}^u)$ on $X$.
\end{enumerate}
\end{Prop}
\begin{proof}
To prove that (R2) implies (R2'), observe that
\begin{align*}
({\rm id}\otimes\widehat{\rho}\otimes\hat{\lambda})({\rm id}\otimes\Delta^u_{\widehat{\GG}})(\wW^\GG)&=\ww_{13}^\GG({\rm id}\otimes\widehat{\rho}\,)(\wW^\GG)_{12}\\
&=(\rho^c\otimes{\rm id})(\mathds{V}^{\widehat{\GG}})_{23}({\rm id}\otimes\widehat{\rho}\,)(\wW^\mathbb{G})_{12}(\rho^c\otimes{\rm id})(\mathds{V}^{\widehat{\GG}})_{23}^*\,.
\end{align*}
Take $\omega\in \bdops(L^2(\GG))_*$. After applying $(\omega\otimes{\rm id}\otimes{\rm id})$ to both sides of the above equation, we obtain
\[
(\widehat{\rho}\otimes\hat{\lambda})\Delta^u_{\widehat{\GG}}((\omega\otimes{\rm id})(\wW^\GG))=(\rho^c\otimes{\rm id})(\mathds{V}^{\widehat{\GG}})(\widehat{\rho}((\omega\otimes{\rm id})(\wW^\GG))\otimes 1)(\rho^c\otimes{\rm id})(\mathds{V}^{\widehat{\GG}})^*\,.
\]
Since the elements $(\omega\otimes{\rm id})(\wW^\GG)$ span a dense subspace of $C^u_0(\widehat{\GG})$, the conclusion follows. The preceding equation defines a covariant representation of $(\widehat{\rho}, \rho^c)$, and the above computation can be easily reversed to yield the remaining implication.
\end{proof}

\begin{Prop}
Let $\GG$ be a locally compact quantum group, let $X$ be a right Hilbert module over a C*-algebra $A$, and let
\[
\pi\colon A\longrightarrow\mathcal{L}(X),\qquad
\rho^c\colon C_0^u(\mathbb{G}^c)\longrightarrow\mathcal{L}(X),
\]
be non-degenerate $*$-homomorphisms. The following are equivalent:
\begin{enumerate}
\item[{\rm (R3)}] The images of $\pi$ and $\rho^c$ commute.
\item[{\rm (R3')}] $(\pi, \rho^c)$ is a covariant representation of the right dynamical system $(A,\widehat{\GG},{\rm triv})$ on $X$.
\end{enumerate}
\end{Prop}

Here, ${\rm triv}\colon A \rightarrow M( A\otimes C_0(\widehat{\GG}))$ denotes the right trivial action given by ${\rm triv}(a) = a\otimes 1$.

\begin{proof}
Assume that (R3) holds. To prove (R3'), we have to show that
\[
(\pi\otimes{\rm id}){\rm triv}(a)=(\rho^c\otimes{\rm id})(\mathds{V}^{\widehat{\GG}})(\pi(a)\otimes 1)(\rho^c\otimes{\rm id})(\mathds{V}^{\widehat{\GG}})^*
\]
for all $a\in A$, which is equivalent to
\[
(\pi(a)\otimes 1)(\rho^c\otimes{\rm id})(\mathds{V}^{\widehat{\GG}})=(\rho^c\otimes{\rm id})(\mathds{V}^{\widehat{\GG}})(\pi(a)\otimes 1).
\]
Let $\omega\in \bdops(L^2(\GG))_*$. Then
\begin{align*}
({\rm id}\otimes\omega)((\pi(a)\otimes 1)(\rho^c\otimes{\rm id})(\mathds{V}^{\widehat{\GG}}))&=\pi(a)\rho^c(({\rm id}\otimes\omega)(\mathds{V}^{\widehat{\GG}}))\\
&=\rho^c(({\rm id}\otimes\omega)(\mathds{V}^{\widehat{\GG}}))\pi(a)\\
&=({\rm id}\otimes\omega)((\rho^c\otimes{\rm id})(\mathds{V}^{\widehat{\GG}})(\pi(a)\otimes 1)).
\end{align*}
Since the above equation holds for any $\omega\in \bdops(L^2(\GG))_*$, the result follows. Again reversing the above computation yields the converse claim.
\end{proof}

We are now prepared to apply Definition~\ref{heirep} and recover the Heisenberg representations promoted in numerous historical contexts. 

\begin{example} \label{relationtomackey}
Consider a Heisenberg $(\GG, A, {\rm triv})$-representation $(X, \rho^c, \hat{\rho}, \pi)$ comprising $A=\mathbb{C}$, $X=H$ a Hilbert space, and $\GG=G$ a classical abelian locally compact group. Then $\rho^c$ and $\widehat{\rho}$ correspond to unitary representations $\widehat{U}$ and $U$ of $\widehat{G}$ and $G$ on $H$ respectively. The conditions (R1) and (R3) are automatically satisfied. Let ${\rm ev}_g\colon C_0(G)\to\mathbb{C}$ and ${\rm ev}_\chi\colon C_0(\widehat{G})\to\mathbb{C}$ be the evaluation functionals at points $g\in G$ and $\chi\in\widehat{G}$ respectively. After applying $({\rm ev}_g\otimes {\rm id}\otimes{\rm ev}_\chi)$ to the defining equation of condition {\rm (R2)}, we obtain
\[
\chi(g)U(g)\widehat{U}(\chi)=\widehat{U}(\chi)U(g),
\]
which is the Mackey's generalization of the Weyl commutation relation~\cite{Mackey}*{Theorem 1}.
\end{example}
\begin{example}
Rieffel refined and generalized Mackey's pioneering work on induced representations of locally compact groups to the setting of of \cs-representation theory (see the introduction to \cite{rinduced} for a nice history), which recovers the Mackey-Stone--von Neumann Theorem treated in Example \ref{relationtomackey} as a special case of Mackey's Imprimitivity Theorem (follow the references in \cite{Rieffel}*{Paragraph Three}). In the latter article, Rieffel defines Heisenberg $G$-modules for an arbitrary locally compact group $G$. These are exactly our Heisenberg $(\GG, A,\alpha)$-representations for $A=\bc$ and $\GG$ the dual to a classical locally compact group. The discerning will notice that the integrand in \cite{Rieffel}*{Equation (3)} has striking similarities to the application of the left regular representation interpreted in the classical setting, an antecedent to our current work. 
\end{example}

\begin{example}
If we take $\GG$ to be a classical abelian locally compact group $G$ and $A$ to be the \cs-algebra of compact operators $\bk$ in Definition~\ref{heirep}, then we recover the definition of the Heisenberg representation given by Huang and Ismert~\cite{HI}. Indeed, they explicitly assume (R1), (R3'), and the Mackey--Weyl commutation relation, which is equivalent to (R2) (see the previous remark).
\end{example}

\begin{example}
In~\cite{modular}, the authors consider parallel generalizations of the Heisenberg representation using tools of nonabelian duality. In each case, the expressions differ in whether $\GG$ is classical (hence derived from a bonafide locally compact group) or universal and dual to classical (i.e the full \cs-algebra of a locally compact group). Correcting for the cosmetics of left and right (co)actions, for classical $\GG$ \cite{modular}*{Definition 4.5} defines a Heisenberg \emph{action}-modular representation in accordance with our conditions (R1), (R2'), and (R3). Likewise, \cite{modular}*{Definition 5.1} defines a Heisenberg \emph{coaction}-modular representation in exactly the same way. Our present work is a proper generalization in that there are numerous locally compact quantum groups which are neither classical nor dual to classical. 
%For example, if we let $\GG=G$ be a classical locally compact group, then $C_0^u(\wh{\GG})$ is isomorphic with the full group \cs-algebra $\cs(G)$ of $G$. Therefore, a coaction of the classical group $G$ in the sense of~\cite{modular} is a coaction of $\wh{\GG}$ in the sense of this paper. The upshot is two divergent definitions 
%Take $\GG$ to a cocommutative locally compact quantum group in Definition~\ref{heirep}. Then the conditions (R1), (R2'), and (R3) are equivalent to the conditions of Definition~5.1 in~\cite{modular}. 
\end{example}
\begin{example}
In~\cite{mrw}, the authors introduced the notion of a~Heisenberg pair which Roy and Timmermann later generalized in~\cite{RT18}. Meyer, Roy, and Woronowicz define Heisenberg pairs in the context of quantum groups derived from manageable multiplicative unitaries (see e.g.~\cite{Wor96}); this concept is closely related to our Heisenberg representations in the sense of Kustermans and Vaes. Compare for instance the parallels between our Proposition~\ref{covcorr} and the argument in~\cite[Theorem~6.2]{RT18}. It is known that Woronowicz description of quantum groups captures the Kustermans and Vaes framework-- that is, given a locally compact quantum group in the sense of this paper, one may construct a manageable multiplicative unitary of the type considered in~\cite{mrw}. 
\end{example}

We end this section with a result which relates our Heisenberg modular representations to dynamics of the crossed product. Much as full crossed products completely capture the dynamics of the associated system, our Heisenberg representations retain the data involved in crossed product constructions. 

\begin{Prop}\label{covcorr}
Let $(\GG,A,\alpha)$ be a left dynamical system and $X$ a Hilbert $A$-module. There is a one-to-one correspondence between Heisenberg representations of $(\GG,A,\alpha)$ on $X$ and covariant representations of the left dynamical system $(\widehat{\GG}^{\rm op},\GG\ltimes A,\widehat{\alpha}^{\rm op})$ on $X$. 
%given by $(X,\rho^c,\widehat{\rho},\pi)\leftrightarrow (\mu,\Pi)$, where $\mu=\rho^c$ and $\Pi=\widehat{\rho}\ltimes\pi$.
\end{Prop}

\begin{proof}
Let $(X,\rho^c,\widehat{\rho},\pi)$ be a Heisenberg representation. Define
\[
\mu: =\rho^c, \qquad\Pi:=\widehat{\rho}\ltimes\pi.
\]
Since $(\widehat{\rho}, \pi)$ is a covariant representation of $(\GG,A,\alpha)$ on $X$ by (R1), $\Pi$ is well defined. We have to prove that
\begin{equation}\label{eq1}
({\rm id}\otimes\Pi)\widehat{\alpha}^{\rm op}(x)=({\rm id}\otimes\mu)(\wW^{\widehat{\GG}^{\rm op}})^*(1\otimes\Pi(x))({\rm id}\otimes\mu)(\wW^{\widehat{\GG}^{\rm op}})
\end{equation}
for all $x\in \GG\ltimes A$. Using~\eqref{crospan}, it suffices to prove the above formula for the elements of the form $j_A(a)$, $a\in A$, and $j_\GG(h)$, $h\in C^u_0(\widehat{\GG})$. Observe that
\begin{align*}
({\rm id}\otimes\mu)(\wW^{\widehat{\GG}^{\rm op}})({\rm id}\otimes\Pi)\widehat{\alpha}^{\rm op}(j_A(a))&=
(\id\otimes \mu)({\wW}^{\wh{\GG}^{\rm op}})
(1\otimes \Pi( j_A(a)))
\\
&=({\rm id}\otimes\rho^c)(\wW^{\widehat{\GG}^{\rm op}})(1\otimes \pi(a))=(1\otimes \pi(a))({\rm id}\otimes\rho^c)(\wW^{\widehat{\GG}^{\rm op}})\\
&=(1\otimes \Pi(j_A(a)))({\rm id}\otimes\mu)(\wW^{\widehat{\GG}^{\rm op}}).
\end{align*}
Here, we used the condition (R3). Similarly, we have that
\begin{align*}
({\rm id}\otimes\Pi)\widehat{\alpha}^{\rm op}(j_\GG(h))&=({\rm id}\otimes\Pi)(\hat{\lambda}\otimes j_\GG)\Delta_{\widehat{\GG}^{\rm op}}^u(h)=({\rm id}\otimes\widehat{\rho})(\hat{\lambda}\otimes{\rm id})\Delta_{\widehat{\GG}^{\rm op}}^u(h)\\
&=\chi((\widehat{\rho}\otimes{\rm id})({\rm id}\otimes\hat{\lambda})\Delta_{\widehat{\GG}}^u(h))\\
&=\chi((\rho^c\otimes{\rm id})(\mathds{V}^{\wh{\GG}})(\wh{\rho}(h)\otimes1)(\rho^c\otimes{\rm id})(\mathds{V}^{\wh{\GG}})^*)\\
&=({\rm id}\otimes\mu)(\wW^{\wh{\GG}^{\rm op}})^*(1\otimes\Pi( j_{\GG}(h)))({\rm id}\otimes\mu)(\wW^{\wh{\GG}^{\rm op}}).
\end{align*}
Here we used the condition (R2') and the formula $\chi(\mathds{V}^{\wh{\GG}})^*=\wW^{\wh{\GG}^{\rm op}}$. We have proved equality \eqref{eq1} for generators, which ends proof in one direction.

Now let $(\mu, \Pi)$ be a covariant representation of $(\wh{\GG}^{\rm op},\GG\ltimes A,\wh{\alpha}^{\rm op})$ on $X$. Define
\[
\rho^c:=\mu,\qquad \widehat{\rho}:=\Pi\circ j_\GG,\qquad \pi:=\Pi\circ j_A\,.
\]
(R1) We have to prove that
\[
({\rm id}\otimes\pi)\alpha(a)=({\rm id}\otimes\wh{\rho})(\wW^\GG)^*(1\otimes\pi(a))({\rm id}\otimes\wh{\rho})(\wW^\GG)
\]
for all $a\in A$. Since the above equation is equivalent to
\[
({\rm id}\otimes\Pi)(({\rm id}\otimes j_A)\alpha(a))=({\rm id}\otimes\Pi)(({\rm id}\otimes j_{\GG})(\wW^\GG)^*(1\otimes j_A(a))({\rm id}\otimes j_{\GG})(\wW^\GG)),
\]
the claim follows from the fact that $(j_\GG,j_A)$ is a covariant representation of $(\GG, A,\alpha)$ on $\GG\ltimes A$.

(R2') We reverse one of the above calculations and obtain
\begin{align*}
(\widehat{\rho}\otimes{\rm id})({\rm id}\otimes\hat{\lambda})\Delta_{\widehat{\GG}}^u(h)&=\chi(({\rm id}\otimes\Pi)\widehat{\alpha}^{\rm op}(j_\GG(h)))\\
&=\chi(({\rm id}\otimes\mu)(\wW^{\wh{\GG}^{\rm op}})^*(1\otimes \Pi(j_{\GG}(h)))({\rm id}\otimes\mu)(\wW^{\wh{\GG}^{\rm op}}))\\
&=(\rho^c\otimes{\rm id})(\mathds{V}^{\wh{\GG}})(\wh{\rho}(h)\otimes1)(\rho^c\otimes{\rm id})(\mathds{V}^{\wh{\GG}})^*.
\end{align*}
(R3) We have that
\[
({\rm id}\otimes\mu)(\wW^{\widehat{\GG}^{\rm op}})({\rm id}\otimes\Pi)\widehat{\alpha}^{\rm op}(j_A(a))=(1\otimes\Pi(j_A(a)))({\rm id}\otimes\mu)(\wW^{\widehat{\GG}^{\rm op}}).
\]
Take $\omega\in \bdops(L^2(\GG))_*$. After applying $(\omega\otimes{\rm id})$ to the above equation and using the definitions of $\rho^c$ and $\pi$, we find
\[
\rho^c((\omega\otimes{\rm id})(\wW^{\widehat{\GG}^{\rm op}}))\pi(a)=\pi(a)\rho^c((\omega\otimes{\rm id})(\wW^{\widehat{\GG}^{\rm op}})).
\]
Since the elements $(\omega\otimes{\rm id})(\wW^{\wh{\GG}^{\rm op}})$ span a dense subspace of $C^u_0(\GG^{c})$, we are done.
\end{proof}

\section{The Covariant Stone--von Neumann Theorem}\label{sec:main}

We now introduce the Schr\"odinger representation of a dynamical system $(\GG, A, \alpha)$, which may be profitably considered as a fundamental example of a representation for the dynamics.

\begin{Def}
Let $(\GG,A,\alpha)$ be a left dynamical system. A {\em Schr\"odinger representation} is a quadruple $(L^2(\GG)\otimes A,\lambda^c\otimes 1,\hat{\lambda}\otimes 1,\alpha)$, where
\[
\lambda^c\colon C^u_0(\GG^c)\longrightarrow C_0(\GG^c),\qquad \hat{\lambda}\colon C^u_0(\wh{\GG})\longrightarrow C_0(\wh{\GG}),
\]
are the canonical surjective $*$-homomorphisms.
\end{Def}

\begin{Prop}
The Schr\"odinger representation is a Heisenberg representation.
\end{Prop}
\begin{proof}
(R1) We have to prove that $(\hat{\lambda}\otimes 1, \alpha)$ is a covariant representation of $(\GG,A,\alpha)$ on $L^2(\GG)\otimes A$, i.e.
\[
({\rm id}\otimes \alpha)\alpha(a)=({\rm id}\otimes(\hat{\lambda}\otimes 1))(\wW^\GG)^*(1\otimes \alpha(a))({\rm id}\otimes(\hat{\lambda}\otimes 1))(\wW^\GG).
\]
Note that the LHS equals $(\Delta_{\GG}\otimes \id)\alpha(a)$ due to the covariance condition. We compute the RHS:
\begin{align*}
({\rm id}\otimes(\hat{\lambda}\otimes 1))(\wW^\GG)^*(1\otimes \alpha(a))({\rm id}\otimes(\hat{\lambda}\otimes 1))(\wW^\GG)&=(\ww^\GG_{12})^*(1\otimes \alpha(a))\ww^\GG_{12}=(\Delta_{\GG}\otimes \id)\alpha(a).
\end{align*}
Here, we used the fact that $\ww^\GG$ implements $\Delta_\GG$.

(R2) It suffices to prove that
\[
\ww_{13}^\GG({\rm id}\otimes\hat{\lambda})(\wW^\GG)_{12}(\lambda^{c}\otimes {\rm id})(\mathds{V}^{\wh{\GG}})_{23}=(\lambda^c\otimes {\rm id})(\mathds{V}^{\wh{\GG}})_{23}({\rm id}\otimes\hat{\lambda})(\wW^\GG)_{12}\,,
\]
which follows from direct calculation
\begin{align*}
\ww_{13}^\GG({\rm id}\otimes\hat{\lambda})(\wW^\GG)_{12}(\lambda^c\otimes {\rm id})(\mathds{V}^{\wh{\GG}})_{23}&=\ww^\GG_{13}\ww^\GG_{12}\vv^{\wh{\GG}}_{23}=({\rm id}\otimes\Delta_{\wh{\GG}})(\ww^\GG)\vv^{\wh{\GG}}_{23}\\
&=\vv^{\wh{\GG}}_{23}\ww^\GG_{12}(\vv^{\wh{\GG}}_{23})^*\vv^{\wh{\GG}}_{23}=\vv^{\wh{\GG}}_{23}\ww^\GG_{12}\\&=(\lambda^c\otimes {\rm id})(\mathds{V}^{\wh{\GG}})_{23}({\rm id}\otimes\hat{\lambda})(\wW^\GG)_{12}\,.
\end{align*}
(R3) The images of $\alpha$ and $\lambda^c\otimes 1$ commute, since $\alpha(A)\subseteq M(C_0(\GG)\otimes A)$ and $(\lambda^c\otimes 1)(C_0^u(\GG^c))\subseteq C_0(\GG^c)\otimes 1$.
\end{proof}

Using the machinery we have gathered so far, we can prove one of the main results of our work.

\begin{Thm}\label{main}
Let $ \Trip{\GG}{A}{\alpha} $ be a left dynamical system, where $A$ is an elementary \cs-algebra, $\GG$ is a regular locally compact quantum group, and $\alpha$ is a maximal action.
% \begin{itemize}
% \item
% $ A $ is an elementary $ C^{\ast} $-algebra.
% \item
% $ \G $ is a strongly regular locally compact quantum group.
% \item
% $ \alpha $ is a \lucaswrong{maximal} coaction of $ \G $ on $ A $.
% \end{itemize}
Then every Heisenberg $ \Trip{\G}{A}{\alpha} $-representation is a multiple of the Schr\"{o}dinger $ \Trip{\G}{A}{\alpha} $-representation.
\end{Thm}

\begin{proof}
%\lucas{This is the sketch of how the proof ``should'' go.}

 Consider a Heisenberg $(\GG, A, \alpha)$-representation $(X,\rho^\c,\widehat{\rho},\pi)$ of $(\GG,A,\alpha)$. Then Proposition \ref{covcorr} yields a covariant representation  $(\rho^\c, \widehat{\rho}\ltimes \pi)$ for $(\widehat{\GG}^{\rm op},\GG\ltimes A,\widehat{\alpha}^{\rm op})$. In turn, the covariant representation $(\rho^\c, \widehat{\rho}\ltimes \pi)$ determines a nondegenerate representation of the iterated crossed product 
\[
\rho^\c \ltimes \widehat{\rho}\ltimes\pi\colon  \widehat{\GG}^\op\ltimes(\GG\ltimes A) \to \mathcal{L}(X).
\]
Since $\alpha$ is maximal, we have an isomorphism $ (\lambda^c\otimes 1)\ltimes (\hat{\lambda}\otimes 1) \ltimes \alpha\colon \widehat{\GG}^\op\ltimes(\GG\ltimes A)\to \bk (L^2(\G))\otimes A$. Abbreviating the isomorphism by $\Phi$, we witness $X$ is an $\bk(L^2(\G) )\otimes A- A$ correspondence with left action coming from the $*$-homomorphism $(\rho^c\ltimes \hat{\rho}\ltimes \pi) \circ \Phi\inv$.

Applying this analysis to the Schr\"odinger representation $(L^2(\GG)\otimes A,\lambda^c\otimes 1,\hat{\lambda}\otimes 1,\alpha),$ we can also see that the correspondence $ L^2(\G) \otimes A$ is an $ \bk(L^2(\G))\otimes A-A$ correspondence. It is in fact an imprimitivity bimodule. To see this, notice that the map $\Phi$ in the previous paragraph agrees with the iterated crossed product $*$-homomorphism associated to the Schr\"odinger representation, so that the left action of $\bk(L^2(\G))\otimes A$ is given by the identity. Since $A$ and $\bk(L^2(\G))\otimes A$ are both simple, the fullness of $L^2(\GG)\otimes A$ is clear, and for the compatibility of the inner products, we compute for elementary tensors
\begin{align*} 
\langle h\otimes c, g\otimes b \rangle \cdot f\otimes a &= (\langle h, g \rangle \otimes \langle c, b \rangle) \cdot f\otimes a = (\theta_{h, g} \otimes cb^* )\cdot f\otimes a = h\langle g, f \rangle \otimes cb^*a \\
&= h\otimes c \cdot (\langle g, f\rangle \otimes \langle b, a\rangle) = h\otimes c(\langle g\otimes b, f\otimes a\rangle). 
\end{align*}
Thus we may apply Theorem \ref{modularsvn} to deduce that $X$ is a multiple of $ L^2(\G)\otimes A$. We suppress the unitary which implements the equivalence $X\cong \bigoplus_S L^2(\GG)\otimes A$, and proceed by  identifying these two modules.  

To finish, disintegrate the integrated form 
\[\rho^\c \ltimes \widehat{\rho}\ltimes\pi\colon  \widehat{\GG}^\op\ltimes(\GG\ltimes A) \to \mc{L}(\bigoplus_S  L^2(\GG)\otimes A),\] 
recovering the covariant representation $(\rho^\c, \hat{\rho}\ltimes\pi)$ of $(\widehat{\G}^\op, \G\ltimes A, \hat{\alpha}^\op).$ Applying Proposition \ref{covcorr}, we recover the Heisenberg $(\G,A, \alpha)$-representation $(\oplus_S \big(L^2(\G)\otimes A\big), \rho^\c, \hat{\rho}, \pi).$ Since each summand is invariant under the $*$-representations, this concludes the proof.
%Then, we apply the correspondence between crossed products and their associated covariant representations, together with Proposition \ref{covcorr}, to recover that the given Heisenberg $\Trip{A}{\G}{\alpha}$-representation $\Quad{\X}{\pi}{U}{V}$ is a multiple of the Schr\"odinger representation $\Quad{A \otimes \L{\G}}{\Br{L_{A} \otimes \Id_{\Co{\G}}} \circ \alpha}{\mathbf{1} \otimes \V_{\G}}{\mathbf{1} \otimes \V_{\Comm{\Dual{\G}}}}.$ This concludes the proof. 
\end{proof}

Let us explicitly state the above result for $A=\mathbb{C}$ which can be thought of as a generalization of the Mackey--Stone--von Neumann theorem~\cite{Mackey} to quantum groups. %Note that maximality of the trivial action of $\mathbb{G}$ on $\mathbb{C}$ is equivalent to strong regularity of $\widehat{\GG}^{\rm op}$ which is equivalent to strong regularity of $\widehat{\GG}$.(?)
\begin{corollary}\label{svnmthm}
Let $\GG$ be a regular locally compact quantum group and assume that the trivial
%the canonical action 
%$\Delta^u_{\widehat{\GG}^\op}$ 
%of $\widehat{\mathbb{G}}^{\rm op}$ on $C^u_0(\widehat{\mathbb{G}})$ is maximal.
of $\GG$ on $\bc$ is maximal. Then every Heisenberg $(\GG,\mathbb{C},{\rm triv})$-representation is a multiple of the Schr\"odinger $(\GG,\mathbb{C},{\rm triv})$-representation $(L^2(\GG),\lambda^c,\hat{\lambda},{\rm triv})$.
\end{corollary}

%\begin{remark}
%Our proof of Theorem~\ref{main} is inspired by Rieffel's original idea~\cite{Rieffel} and in the case of Corollary~\ref{svnmthm} one could adapt the proof of~\cite[Theorem~6.2]{RT18} instead of using Proposition~\ref{covcorr}.
%\end{remark}

In light of Theorem \ref{main}, we introduce a new property for a left dynamical system $(\GG, A, \alpha)$, which is a quantum-dynamical analog of the behavior described in \cite{Rieffel}*{Theorem} and \cite{wcrossed}*{Theorem 4.29}. 

\begin{Def}
A dynamical system $(\GG, A, \alpha)$ is called Stone--von Neumann unique provided that every Heisenberg $(\GG,A,\alpha)$-representation is a multiple of the Schr\"{o}dinger $(\GG,A,\alpha)$-representation.
\end{Def}
By Theorem \ref{main} if $\GG$ is regular, $A$ is elementary and $\alpha$ is maximal, then $(\GG, A, \alpha)$ is Stone--von Neumann unique. In the other direction, given a dynamical system $(\GG, A, \alpha)$ which possesses Stone--von Neumann uniqueness, one may investigate what can be said of the elements of the dynamical system. We spend the remainder of this article pursuing these questions. The following key theorem is the main result of the paper, which relates the multiplicity result from Theorem~\ref{main} to the properties of the spectra and provides the essential tool for analyzing Stone-von Neumann unique dynamical systems.

\begin{Thm}\label{prop1}
Let $\GG$ be a locally compact quantum group and $A$ a $C^*$-algebra admitting a faithful state. Assume that there is a dynamical  $(\GG,A,\alpha)$ which is Stone--von Neumann unique. Then the spectrum of $\wh{\GG}^\op\ltimes (\GG\ltimes A)$ consists of a single point.
\end{Thm}

\begin{proof}
We begin with a Stone--von Neumann unique dynamical system $(\GG, A, \alpha)$, and must show that any two irreducible representations $\theta,\tilde{\theta}$ of $\wh{\GG}^\op\ltimes (\GG\ltimes A)$ on Hilbert spaces $H_\theta,H_{\tilde{\theta}}$ are equivalent. First, consider $\theta\otimes 1$ as a representation of $\wh{\GG}^\op\ltimes (\GG\ltimes A)$ on Hilbert-$A$ module $H_\theta\otimes A$, similarly $\tilde{\theta}\otimes 1$. Using the canonical covariant representation of $(\wh{\GG}^\op,\GG\ltimes A,\widehat{\alpha}^\op)$
\[
j_{\wh{\GG}^\op}\colon C_0^u(\GG^{c})\rightarrow 
M(\wh{\GG}^\op\ltimes (\GG\ltimes A)),\quad
j_{\GG\ltimes A}\colon 
\GG\ltimes A\rightarrow
M(\wh{\GG}^\op\ltimes (\GG\ltimes A))
\]
and then composing with $\theta\otimes 1$ (extended uniquely to a strictly continuous map on the multiplier algebra), we obtain a covariant representation of $(\wh{\GG}^\op,\GG\ltimes A,\wh{\alpha}^\op)$ on $H_\theta\otimes A$. By Proposition \ref{covcorr}, we recover the Heisenberg representation $(H_\theta\otimes A, \rho^c_\theta,\wh{\rho}_\theta,\pi_\theta)$ of $(\GG,A,\alpha)$. More concretely,
\[
\rho^c_\theta =(\theta \circ j_{\wh{\GG}^\op})\otimes 1,\quad
\wh{\rho}_\theta= (\theta \circ j_{\GG\ltimes A}\circ j_{\GG})\otimes 1,\quad
\pi_\theta=(\theta\circ j_{\GG\ltimes A}\circ j_{A})\otimes 1,
\]
where
\[
j_{\GG}\colon C_0^u(\wh{\GG})\rightarrow 
M(\GG\ltimes A),\quad
j_A\colon A\rightarrow M(\GG\ltimes A)
\]
are the canonical maps. We can perform a similar construction for $\tilde{\theta}$. By our assumption, there is a set $S_\theta$ so that $(H_\theta\otimes A, \rho^c_\theta,\widehat{\rho}_\theta,\pi_\theta)$ is a $|S_\theta|$-multiple of the Schr{\"o}dinger representation. Similarly, ${(H_{\tilde{\theta}}\otimes A, \rho^c_{\tilde{\theta}},\widehat{\rho}_{\tilde{\theta}},\pi_{\tilde{\theta}})}$ is a Heisenberg representation which is equivalent to a $|S_{\tilde{\theta}}|$-multiple of the Schr{\"o}dinger representation. Let $S=S_\theta\times S_{\tilde{\theta}}\times \mathbb{N}$. By cardinal arithmetic, we have $|S\times S_{\theta}| =|S|=|S\times S_{\tilde{\theta}}|$, hence $|S|$-multiples of $(H_\theta\otimes A, \rho^c_\theta,\widehat{\rho}_\theta,\pi_\theta)$ and $(H_{\tilde{\theta}}\otimes A, \rho^c_{\tilde{\theta}},\widehat{\rho}_{\tilde{\theta}},\pi_{\tilde{\theta}})$ are equivalent. Let
\[
U\colon \bigoplus_{S} H_\theta\otimes A \rightarrow \bigoplus_{S} H_{\tilde{\theta}}\otimes A
\]
be the corresponding isomorphism which implements this equivalence. \\

By our assumption, there is a faithful state $\varpi$ on $A$, so let $K_\varpi$ be the corresponding GNS Hilbert space and $\Lambda_\varpi$ the corresponding GNS map. Let us consider
\begin{equation}\label{eq1.1}
 \id\otimes \Lambda_\varpi\colon H_{ \theta}\otimes A\rightarrow H_{\theta}\otimes K_\varpi
\end{equation}
which on a linearly dense set acts via $\xi\otimes  a\mapsto \xi\otimes \Lambda_\varpi(a)$. Take an arbitrary linear combination from this set $\sum_{i} \xi_i\otimes a_{i}\in H_{\theta}\odot A$ and calculate
\[\begin{split}
\bigl\|
(\id\otimes \Lambda_\varpi)\bigl(
\sum_{i}
\xi_i\otimes a_i
\bigr)\bigr\|^2
&=
\sum_{i,j}
\la \xi_i|\xi_{j}\ra 
\la \Lambda_\varpi(a_{i}) |
\Lambda_{\varpi}(a_{j})\ra =
\sum_{i,j}
\la \xi_i|\xi_{j}\ra 
\varpi(a_{i}^* a_{j})\\
&=
\varpi\bigl(
\bigl \la \sum_{i} \xi_i \otimes a_{i},
 \sum_{i} \xi_i \otimes a_{i}\bigr\ra\bigr)\le
 \bigl\|
\bigl \la \sum_{i} \xi_i \otimes a_{i},
 \sum_{i} \xi_i \otimes a_{i}\bigr\ra
\bigr\|\\
&=\bigl\|
\sum_{i}
\xi_i \otimes a_{i}
\bigr\|^2.
\end{split}\]
It follows that \eqref{eq1.1} is a well defined contraction. Next we claim that
\begin{equation}\label{eq2}
\la (\id\otimes \Lambda_\varpi)(x)|
(\id\otimes  \Lambda_\varpi)(x')\ra=
\varpi(\la x,x'\ra ).
\end{equation}
for $x,x'\in H_{\theta}\otimes A$. Both sides are sesquilinear and continuous in $x,x'$, hence it is enough to check \eqref{eq2} on a linearly dense set: take $x=\xi\otimes a,x'=\xi'\otimes a'$. Then
\[
\la (\id\otimes \Lambda_\varpi)(x)|
( \id\otimes \Lambda_\varpi)(x')\ra=
\la \xi \otimes  \Lambda_\varpi(a)|
\xi' \otimes 
\Lambda_{\varpi}(a')\ra =
\la \xi|\xi'\ra   \varpi(a^* a')=
\varpi(\la x,x'\ra ).
\]
Note that from \eqref{eq2} follows that $\id\otimes \Lambda_\varpi$ is injective. A similar reasoning can be performed for $\tilde{\theta}$ -- we will abuse the notation and denote the corresponding map also by $\id\otimes \Lambda_\varpi$. Next let us define a linear map between Hilbert spaces
\[
\tilde{U}\colon \bigoplus_{S} H_\theta\otimes K_\varpi\rightarrow
\bigoplus_{S}  H_{\tilde{\theta}}\otimes K_\varpi.
\]
First, on a dense subspace $\bigoplus_S (\id\otimes \Lambda_\varpi)(  H_\theta\odot A)$, we set
\begin{equation}\label{eq3}
\tilde{U}\bigl( ((\id\otimes \Lambda_\varpi)\Omega_i)_{i\in S}\bigr)=
\bigl( (\id\otimes \Lambda_\varpi)(\zeta_i)\bigr)_{i\in S}\textnormal{ where }
U 
\bigl( (\Omega_i)_{i\in S}\bigr)=(\zeta_i)_{i\in S}.
\end{equation}
Since $\varpi$ is faithful, $\tilde{U}$ is well defined on a dense subset. Let us argue that it is isometric. Take $((\id\otimes \Lambda_\varpi)\Omega_i)_{i\in S}$ and $\zeta_i$ as above. Then using \eqref{eq2} we have
\[\begin{split}
\bigl\| \tilde{U}\bigl((
(\id \otimes \Lambda_\varpi)\Omega_i)_{i\in S}\bigr)\bigr\|^2
&=
\bigl\| \bigl(( \id\otimes \Lambda_{\varpi})(\zeta_i)\bigr)_{i\in S}\bigr\|^2=
\sum_{i\in S} \|(\id\otimes \Lambda_\varpi)(\zeta_i)\|^2\\
&=
\sum_{i\in S} \varpi(\la \zeta_i,\zeta_i\ra )=
\varpi\bigl( \sum_{i\in S} \la \zeta_i,\zeta_i\ra \bigr)=
\varpi\bigl( \bigl\la (\zeta_i)_{i\in S},
(\zeta_i)_{i\in S}\bigr\ra \bigr)\\
&=
\varpi\bigl(
\bigl\la 
U 
\bigl( (\Omega_i)_{i\in S}\bigr),
U 
\bigl( (\Omega_i)_{i\in S}\bigr)\bigr\ra\bigr)=
\varpi\bigl(
\bigl\la 
(\Omega_i)_{i\in S},(\Omega_i)_{i\in S}\bigr\ra\bigr)\\
&=
\sum_{i\in S} 
\varpi(\la \Omega_i,  \Omega_i\ra )
=
\sum_{i\in S} 
\la (\id\otimes \Lambda_\varpi)\Omega_i | (\id\otimes \Lambda_\varpi)\Omega_i\ra\\ 
&=
\bigl\|
\bigl(
(\id\otimes \Lambda_\varpi)\Omega_i
\bigr)_{i\in S}
\bigr\|^2.
\end{split}\]
This shows that $\tilde{U}$ is an isometry. Clearly it has dense image; consequently, it is unitary.\\

Next, we check that $\tilde{U}$ is an intertwiner between two representations of $M( \wh{\GG}^\op\ltimes (\GG\ltimes A))$, the first being $\bigoplus_S \theta\otimes 1$ on $\bigoplus_S H_{\theta}\otimes K_\varpi$ and the other $\bigoplus_{S}  \tilde{\theta}\otimes 1$ on  $\bigoplus_S H_{\tilde{\theta}}\otimes K_\varpi$. From \eqref{eq3} it follows that
\[
\tilde{U}\bigl(\bigoplus_S ( \id\otimes \Lambda_\varpi)\bigr)=
\bigl(\bigoplus_S (\id\otimes\Lambda_\varpi)\bigr) 
U\quad\textnormal{ on } \bigoplus_S H_\theta\otimes A.
\]
We need to check that 
\[
\tilde{U} \bigl(\bigoplus_S \theta(x)\otimes 1\bigr)=
\bigl(\bigoplus_S \tilde{\theta}(x)\otimes 1\bigr)\tilde{U}
\]
for $x\in M(\wh{\GG}^\op\ltimes (\GG\ltimes A))$. However, since $M(\wh{\GG}^\op\ltimes (\GG\ltimes A))$ is generated by the images of $C_0^u(\GG^c), C_0^u(\wh{\GG}), A$ it is enough to consider $x$ of the form $j_{\wh{\GG}^\op}(y)$, $j_{\GG\ltimes A}\circ j_{\GG}(z)$ and $j_{\GG\ltimes A}\circ j_A(a)$. Take $a\in A$.

\[\begin{split}
&\quad\;
\tilde{U}
\bigl(\bigoplus_S  \theta(j_{\GG\ltimes A}\circ j_A(a))\otimes 1\bigr)
(\xi_i\otimes  \Lambda_\varpi(a_i))_{i\in S}=
\tilde{U}
(\theta(j_{\GG\ltimes A}\circ j_A(a))\xi_i \otimes \Lambda_\varpi(a_i))_{i\in S}\\
&=
\tilde{U}
\bigl( \bigoplus_S ( \id\otimes \Lambda_\varpi)\bigr)
( \theta(j_{\GG\ltimes A}\circ j_A(a))\xi_i \otimes a_i)_{i\in S}=
\bigl( \bigoplus_S (\id\otimes \Lambda_\varpi)\bigr)
U
( \theta(j_{\GG\ltimes A}\circ j_A(a))\xi_i \otimes a_i)_{i\in S}\\
&=
\bigl( \bigoplus_S ( \id\otimes \Lambda_\varpi)
(\tilde{\theta}(j_{\GG\ltimes A}\circ j_A(a))\otimes 1)
\bigr)
U
(\xi_i\otimes a_i)_{i\in S}\\
&=
\bigl( \bigoplus_S 
 \tilde{\theta}(j_{\GG\ltimes A}\circ j_A(a))\otimes 1\bigr)
\bigl(
\bigoplus_S ( \id\otimes \Lambda_\varpi)
\bigr)
U
(\xi_i\otimes  a_i)_{i\in S}\\
&=
\bigl( \bigoplus_S 
  \tilde{\theta}(j_{\GG\ltimes A}\circ j_A(a))\otimes 1\bigr)
\tilde{U}
(\xi_i\otimes \Lambda_{\varpi}(a_i))_{i\in S}.
\end{split}\]
The remaining conditions can be checked in a completely analogous way. This shows that $\bigoplus_S \theta\otimes 1$ and  $\bigoplus_{S} \tilde{\theta}\otimes 1$ are equivalent. Consequently, they are multiples of $\theta$ and $\tilde{\theta}$ which are equivalent. But since $\theta,\tilde{\theta}$ are assumed to be irreducible, it follows that they are equivalent themselves (\cite[Proposition 5.3.3]{DixmierC}). In particular, they are equivalent as representations of $\wh{\GG}^\op\ltimes (\GG\ltimes A)$ and consequently this $C^*$-algebra has spectrum consisting of one point. 
\end{proof}

Let us note here that the above proposition applies to any separable $C^*$-algebra, since such a~$C^*$-algebra admits a faithful state. 

With this in hand, we will now see that a complete characterization of coefficient algebras $A$ boils down to Naimark's problem \cite{naimark}. To this end, let us make a general observation that for the trivial action $(\GG, A, \operatorname{triv})$, we have
\begin{equation}\label{eq4}
\wh{\GG}^\op\ltimes (\GG\ltimes A)\cong (\wh{\GG}^\op\ltimes C_0^u(\wh{\GG}))\otimes_{\textnormal{max}} A.
\end{equation}

\begin{corollary}\label{convelem}
Let $\GG$ be a locally compact quantum group and $A$ a separable \cs-algebra, and assume that $(\GG, A, \operatorname{triv})$ enjoys Stone--von Neumann uniqueness. Then $A$ is an elementary \cs-algebra.
\end{corollary}

\begin{proof}
Note that since $A$ is separable, it has a faithful state. Let $\vartheta_2,\tilde{\vartheta_2}$ be irreducible representations of $A$ and $\vartheta_1$ an irreducible representation of $\wh{\GG}^\op\ltimes C_0^u(\wh{\GG})$. Then we can form a representation $\vartheta_1\otimes \vartheta_2$ of $(\wh{\GG}^\op\ltimes C_0^u(\wh{\GG}))\otimes_{\textnormal{max}} A$ on $H_{\vartheta_1}\otimes H_{\vartheta_2}$; same for $\vartheta_1\otimes \tilde{\vartheta}_2$. They are irreducible, hence by the isomorphism \eqref{eq4} above  and Theorem \ref{prop1}, there is a unitary $v\colon H_{\vartheta_1}\otimes H_{\vartheta_2}\rightarrow H_{\vartheta_1}\otimes H_{\tilde{\vartheta}_2}$ such that $v(\vartheta_1\otimes \vartheta_2)(\cdot) = (\vartheta_1\otimes \tilde{\vartheta}_2)(\cdot)v$. Extend these representations to multiplier algebras, and consider representations $A\ni a\mapsto (\vartheta_1\otimes \vartheta_2)(1\otimes a)=1\otimes \vartheta_2(a)\in \bdops(H_{\vartheta_1}\otimes H_{\vartheta_2})$, and same for $\tilde{\vartheta}_2$. These are multiples of $\vartheta_2,\tilde{\vartheta}_2$ which are equivalent via the unitary $v$. Since $\vartheta_2,\tilde{\vartheta}_2$ are irreducible, we can use again \cite[Proposition 5.3.3]{DixmierC} to deduce that $\vartheta_2,\tilde{\vartheta}_2$ are equivalent. We assume that $A$ is separable, hence \cite[Exercise 4.7.2]{DixmierC} tells us that $A$ is elementary.
\end{proof}

We also observe that separability is an important obstruction to a general characterization of algebras with one-point spectrum. Akemann and Weaver demonstrate in \cite{awCounterexample} that the existence of nonelementary \cs-algebras with one-point spectrum and $ \aleph_{1} $ generators is independent of ZFC.

\begin{corollary}\label{convmax}
Let $(\GG,A,\alpha)$ be a left dynamical system, where $\GG$ is a regular locally compact quantum group and $A$ a \cs-algebra which admits a faithful state. If $(\GG, A, \alpha)$ is subject to Stone--von Neumann uniqueness, then $ \alpha$ is a maximal action. 
\end{corollary}

\begin{proof}
We have a canonical surjective $*$-homomorphism $\rho\colon \wh{\GG}^\op\ltimes (\GG\ltimes A)\rightarrow \bk (L^2(\GG))\otimes A$, and we want to show that this map is injective. Assume otherwise, i.e.~assume that $\ker(\rho)$ is a nontrivial ideal in  $\wh{\GG}^\op\ltimes (\GG\ltimes A)$. Then we can form the quotient $C^*$-algebra $(\wh{\GG}^\op\ltimes (\GG\ltimes A))/\ker(\rho)$ and choose its irreducible representation $\theta\colon (\wh{\GG}^\op\ltimes (\GG\ltimes A))/\ker(\rho)\rightarrow \bdops(H_\theta)$. Composing with the quotient map, we obtain an irreducible representation $\tilde{\theta}\colon \wh{\GG}^\op\ltimes (\GG\ltimes A)\rightarrow \bdops(H_\theta)$ such that $\ker(\rho)\subseteq \ker(\tilde{\theta})$. On the other hand, since $\ker(\rho)$ is nontrivial, we can choose $x\in\ker(\rho)\setminus\{0\}$ and an irreducible representation $\tilde{\theta}'\colon \wh{\GG}^\op\ltimes (\GG\ltimes A) \rightarrow \bdops(H_{\tilde{\theta}'})$ such that $\tilde{\theta}'(x)\neq 0$. This shows that $\tilde{\theta},\tilde{\theta}'$ have different kernels, in particular are not equivalent -- this gives us a contradiction with Theorem \ref{prop1}. Thus $\rho$ is an isomorphism, and the action $\alpha$ is maximal.
\end{proof}

\begin{remark}

Corollaries~\ref{convelem} and~\ref{convmax} indicate that the assumptions of Theorem~\ref{main}, namely elementariness of the coefficient algebra and maximality of the action, are not conditions designed to simplify arguments. On the contrary, they appear necessary. 
\end{remark}

As a final application, we prove the converse to Corollary~\ref{svnmthm} and show that Stone--von Neumann uniqueness of $(\mathbb{G},\mathbb{C},{\rm triv})$, where $\GG$ is regular, implies maximality of the trivial action of $\mathbb{G}$ on $\bc$.

\begin{corollary}\label{fullconverse}
Let $\GG$ be a regular locally compact quantum group. The following properties are equivalent:
\begin{enumerate}
\item $(\GG, \bc, \operatorname{triv})$ is Stone--von Neumann unique,
\item the trivial action of $\GG$ on $\bc$ is maximal,
\item $\GG$ is strongly regular.
\end{enumerate}
\end{corollary}

\begin{proof}
If $(\GG,\bc,{\rm triv})$ is Stone--von Neumann unique, then {\rm triv} is maximal by Corollary~\ref{convmax}. The converse implication is the content of Corollary~\ref{svnmthm}. Equivalence of $(2)$ and $(3)$ was proved in Proposition \ref{maxtrivdual}.
\end{proof}

%\begin{corollary}\label{fullconverse}
%Let $\GG$ be a regular locally compact quantum group. Then $(\GG, \bc, \operatorname{triv})$ is Stone--von Neumann unique if and only if the trivial action of $\GG$ on $\bc$ is maximal. 
%\end{corollary}

%\begin{proof}
%If $(\GG,\bc,{\rm triv})$ is Stone--von Neumann unique, then {\rm triv} is maximal by Corollary~\ref{convmax}. The converse implication is the content of Corollary~\ref{svnmthm}.
%follows from Corollary~\ref{svnmthm} combined with Lemma~\ref{maxtrivdual}.
%\end{proof}

%\begin{corollary}
%Let $\GG$ be a regular locally compact quantum group. If $(\GG,\bc,{\rm triv})$ is Stone--von Neumann unique, then $\GG$ is strongly regular.
%\end{corollary}
%\begin{proof}
%First, note that Corollary~\ref{fullconverse} together with Proposition~\ref{maxtrivdual} gives us an isomorphism $\GG^\op \ltimes C_0^u(\GG)\cong  \bk (L^2(\GG))$. Let $j_{\GG^\op}, j_{C_0(\GG)}$ be the canonical maps associated with the crossed product $\GG^\op\ltimes C_0(\GG)$. Observe that $( j_{{\GG}^\op},  j_{C_0(\GG)}\circ\lambda)$ is a covariant representation of the system $(\GG^\op, C_0^u(\GG),$ $(\lambda\otimes \id) \circ\Delta^u_{\GG^\op})$ on $\GG^\op\ltimes C_0(\GG)$ whose integrated form determines a (nonzero) surjection of $\bk(L^2(\GG))$. Since $\bk(L^2(\GG))$ is simple, it has to be an isomorphism. We conclude that $\GG^\op\ltimes C_0(\GG^\op)$ comprises compact operators so that by definition, $\GG^\op$ is strongly regular. This property is equivalent to strong regularity of $\GG$, which ends the proof.
%\end{proof}

\end{document}